\documentclass[11pt,reqno]{amsart}
\usepackage[utf8]{inputenc}

\usepackage[margin=1in]{geometry}
\usepackage{amsmath}
\usepackage{amsfonts}
\usepackage{amssymb}
\usepackage{amsthm}
\usepackage{enumitem}
\usepackage{mathrsfs}
\usepackage{mathtools}
\usepackage{tikz-cd}
\usepackage{enumitem}

\usepackage[colorlinks=true,hyperindex, linkcolor=magenta, pagebackref=false, citecolor=cyan,pdfpagelabels]{hyperref}

\newtheorem{theorem}{Theorem}[section]
\newtheorem{prop}[theorem]{Proposition}
\newtheorem{lemma}[theorem]{Lemma}
\newtheorem{cor}[theorem]{Corollary}

\theoremstyle{definition}
\newtheorem{definition}[theorem]{Definition}
\newtheorem{example}[theorem]{Example}
\newtheorem{remark}[theorem]{Remark}

\renewcommand{\rm}{\mathrm}

\newcommand{\Spec}{\mathrm{Spec} \,}

\newcommand{\vp}{\varphi}

\newcommand{\bF}{{\mathbf F}}
\newcommand{\bG}{{\mathbf G}}

\newcommand{\bZ}{{\mathbf Z}}

\newcommand{\cO}{\mathcal{O}}

\newcommand{\cX}{\mathcal{X}}
\newcommand{\cY}{\mathcal{Y}}

\newcommand{\sD}{{\mathscr D}}

\newcommand{\sN}{{\mathscr N}}
\newcommand{\sO}{{\mathscr O}}

\newcommand{\sR}{{\mathscr R}}

\newcommand{\sU}{{\mathscr U}}

\newcommand{\fg}{\mathfrak{g}}

\newcommand{\fm}{\mathfrak{m}}

\newcommand{\fp}{\mathfrak{p}}

\newcommand{\ft}{\mathfrak{t}}
\newcommand{\fu}{\mathfrak{u}}

\DeclareMathOperator{\Ad}{Ad}

\DeclareMathOperator{\chara}{char}

\DeclareMathOperator{\diag}{diag}

\DeclareMathOperator{\GL}{GL}

\DeclareMathOperator{\Hom}{Hom}

\DeclareMathOperator{\id}{id}

\DeclareMathOperator{\Lie}{Lie}

\DeclareMathOperator{\pgl}{\mathfrak{pgl}}
\DeclareMathOperator{\PGL}{PGL}

\DeclareMathOperator{\SL}{SL}

\DeclareMathOperator{\Sp}{Sp}

\DeclareMathOperator{\Sym}{Sym}

\DeclareMathOperator{\Transp}{Transp}

\newcommand{\wtil}[1]{\wtil}

\newcommand{\ov}[1]{\overline{#1}}

\title{Centralizers of sections of a reductive group scheme}
\author{Sean Cotner}

\begin{document}
\bibliographystyle{halpha-abbrv}

\begin{abstract}
This paper proves a number of flatness results for centralizers of sections of a reductive group scheme over a general base scheme. To this end, we establish relative versions of the Jordan decomposition. Using our results, we obtain a canonical flattening stratification for the universal centralizer of a simply connected semisimple group scheme over a base of good characteristic. We also investigate the structure of centralizers and conjugacy classes of unipotent and nilpotent sections over general bases.
\end{abstract}

\maketitle

\section{Introduction}

\subsection{Overview}

Let $S$ be a scheme, let $G$ be a reductive group scheme over $S$, and let $\fg$ be its Lie algebra. The aim of this paper is to investigate the centralizer schemes $Z_G(g)$ and $Z_G(X)$ of sections $g \in G(S)$ and $X \in \fg(S)$, and in particular to study conditions under which these are $S$-flat or $S$-smooth. There are a number of obstructions to the flatness of (say) $Z_G(g)$; for instance, if the function $s \mapsto \dim Z_{G_s}(g_s)$ is not locally constant on $S$, then $Z_G(g)$ has no hope of being flat. Thus we will call a section $g \in G(S)$ \textit{pure} provided that the above function is locally constant. \smallskip

Flatness can fail even if $g$ is pure; for instance, if $S$ is an Artin local scheme then every section is pure, but there are many non-central sections $g \in G(S)$ such that $g_s = 1$, and in such cases $Z_G(g)$ is usually not flat. Thus if we want any flatness result for $Z_G(g)$ involving only general fibral hypotheses on $g \in G(S)$ (like purity), we must assume that $S$ is reduced.\smallskip

Still flatness can fail; for instance, if $S$ is the spectrum of a discrete valuation ring, then the special fiber of $Z_G(g)$ can have components with no integral points. For examples, see Remark~\ref{remark:possible-extensions-of-flatness-results}. In some sense, these examples come from the fact that the Jordan decomposition for pure sections need not vary nicely among fibers. Thus we say that a section $g \in G(S)$ has \textit{pure semisimple part} if the dimension of the centralizer of the semisimple part of $g_{\overline{s}}$ in the Jordan decomposition is locally constant in $s \in S$. (This terminology will be justified by Theorem~\ref{theorem:intro-jordan}.) Ultimately, we can prove the following theorem. For the definitions of goodness and strong purity, see Sections~\ref{subsection:pretty-good} and \ref{subsection:pure}.

\begin{theorem}[Theorem~\ref{theorem:flat-pure-centralizer}]\label{theorem:intro-flat-pure-centralizer}
Let $S$ be a reduced scheme and let $G$ be a reductive group scheme over $S$ with Lie algebra $\fg$. Assume that $\chara k(s)$ is good for $G_s$ for every $s \in S$ and $|\pi_1(\sD(G))|$ is invertible on $S$ (where $\sD(G)$ is the derived group of $G$).
\begin{enumerate}
    \item\label{item:intro-pure-lie-algebra-centralizer} If $X \in \fg(S)$ is pure and has pure semisimple part, then $Z_G(X)$ is flat and $Z_G(X)/Z(G)$ is smooth.
    \item\label{item:intro-pure-group-centralizer} If $g \in G(S)$ is strongly pure and has pure semisimple part, then $Z_G(g)$ is flat and $Z_G(g)/Z(G)$ is smooth.
\end{enumerate}
\end{theorem}

In the fiberwise nilpotent case, Theorem~\ref{theorem:intro-flat-pure-centralizer}(\ref{item:intro-pure-lie-algebra-centralizer}) was claimed in \cite[5.2]{McNinch-relative-centralizer} but the proof contained a flaw, as pointed out in \cite[Rmk.\ 4.20]{Booher}. Special cases of this result were obtained using an entirely different method in \cite[Prop.\ 4.17]{Booher}. A recent preprint of Hardesty \cite{Hardesty} proves the fiberwise nilpotent case of the first statement under the slightly more stringent requirement that the fibers of $G$ are \textit{geometrically standard}; see \cite[1.6]{Hardesty}. (This amounts to the further assumption that $Z(G)$ is smooth by Lemma~\ref{lemma:characterization-of-pretty-good-primes} and \cite[Thm.\ 5.2]{Herpel}.) The method of proof in \cite{Hardesty} is completely different from ours, involving a careful analysis of the component groups of the fibers of $Z_G(X)$. We outline our proof in Section~\ref{subsection:outline}.\smallskip

Previous cases of Theorem~\ref{theorem:intro-flat-pure-centralizer} have mainly focused around the case that $X$ and $g$ is fiberwise regular; see \cite[Exp.\ XIV, Cor.\ 3.13, Thm.\ 3.18]{SGA3II}, \cite[Prop.\ 3.3.5]{Riche-universal-centralizer}, and \cite[Thm.\ 4.2.8]{Bouthier-Cesnavicius} for the Lie algebra case and \cite[Thm.\ 1.2]{Integral-Springer} for the group case.

The remaining statements of Theorem~\ref{theorem:intro-flat-pure-centralizer} follow from the fiberwise nilpotent case via two new ingredients: an integral Springer isomorphism \cite[Thm.\ 5.1]{Integral-Springer} and the following relative version of the Jordan decomposition.

\begin{theorem}[Theorem~\ref{theorem:relative-jordan-decomposition-over-normal-base}]\label{theorem:intro-jordan}
Let $S$ be a normal integral scheme with generic point $\eta$ and let $G$ be a reductive $S$-group scheme.
\begin{enumerate}
    \item Let $g \in G(S)$ be a strongly pure section of $G$ with pure semisimple part, and for every point $s \in S$ let $g_{\overline{s}} = t_s u_s$ be the Jordan decomposition of $g_{\overline{s}}$. Suppose that the Jordan decomposition of $g_{\overline{\eta}}$ is $k(\eta)$-rational. Then there exist unique sections $t, u \in G(S)$ such that $g = tu$ induces the Jordan decomposition on every $S$-fiber.
    \item Suppose that no residue characteristic of $S$ is a torsion prime for $G$. Let $X \in \fg(S)$ be a pure section of $\fg$ with pure semisimple part, and for every point $s \in S$ let $X_{\overline{s}} = (X_{\rm{ss}})_s + (X_{\rm{n}})_s$ be the Jordan decomposition of $X_{\overline{s}}$. Suppose that the Jordan decomposition of $X_{\overline{\eta}}$ is $k(\eta)$-rational. Then there exist unique sections $X_{\rm{ss}}, X_{\rm{n}} \in \fg(S)$ such that $X = X_{\rm{ss}} + X_{\rm{n}}$ induces the Jordan decomposition on every $S$-fiber.
\end{enumerate}
\end{theorem}

Note that when $S$ is the spectrum of a field, this result has no content. For more general $S$, however, there are two aspects of Theorem~\ref{theorem:intro-jordan} which may be surprising. First, it asserts (under its hypotheses) that the Jordan decomposition on the generic fiber is actually \textit{integral}, i.e., it extends to a decomposition over $S$. This can fail in general; see for instance \cite[paragraph following Question 3.5.3]{DeB02}. Second, it asserts that the fibers of this decomposition are the Jordan decompositions of the fibers. In particular, rationality of the Jordan decomposition on the generic fiber implies rationality of the Jordan decomposition on \textit{all} fibers. The proof of Theorem~\ref{theorem:intro-jordan} bears some similarity to the proof of the classical Jordan decomposition, but it is much less direct.\smallskip

Theorem~\ref{theorem:intro-flat-pure-centralizer} is enough to prove the existence of a canonical $S$-flat flattening stratification for the universal centralizer of a simply connected semisimple group scheme $G$; see Corollary~\ref{cor:flat-strat}. In Theorem~\ref{theorem:flat-pure-centralizer-semidirect-product}, we also establish detailed structural results for nilpotent and unipotent centralizers. Using these structural results and known cases of the Grothendieck--Serre conjecture, we establish the following conjugacy result for pure unipotent sections generalizing \cite[Thm.\ 1.6.1(c)]{McNinch-local-fields}, which will play an important technical role in some arithmetic applications in \cite{booher-tang}.

\begin{theorem}[Theorem~\ref{theorem:grothendieck-serre}]\label{theorem:intro-grothendieck-serre}
Let $A$ be a henselian local ring which is either regular or a valuation ring, and let $G$ be a reductive $A$-group scheme such that the residue characteristic is pretty good for $G$. Let $u, v \in G(A)$ be pure fiberwise unipotent sections. The following are equivalent.
\begin{enumerate}
    \item $u$ and $v$ are $G(A)$-conjugate.
    \item $u_s$ and $v_s$ are $G(k(s))$-conjugate.
    \item $u_\eta$ and $v_\eta$ are $G(k(\eta))$-conjugate.
\end{enumerate}
\end{theorem}

Finally, we conclude this overview with the following result which we prove along the way.

\begin{cor}[Corollary~\ref{cor:unip-comp-gp}]\label{cor:intro-unip-comp-gp}
Let $k$ be a field, and let $G$ be a connected reductive $k$-group such that $\chara k$ is good for $G$. If $u \in G(k)$ is unipotent, then the only primes dividing the order of $Z_G(u)/Z_G(u)^0Z(G)$ are bad for $G$. In particular, the order of $Z_G(u)/Z_G(u)^0Z(G)$ always divides a power of $30$.
\end{cor}

Corollary~\ref{cor:intro-unip-comp-gp} has been known for a long time, as one can see by examining the tables of centralizers in \cite{Liebeck-Seitz}. However, our proof gives a conceptual explanation independent of any case-checking.

\subsection{Outline of the proof}\label{subsection:outline}

We now briefly outline the proof of Theorem~\ref{theorem:intro-flat-pure-centralizer} in the case of sections of $\fg$. This is similar to (but more complicated than) the proof of \cite[Thm.\ 3.10]{Integral-Springer}. First, a number of reductions allow us to assume that $S$ is the spectrum of a complete discrete valuation ring $A$ with algebraically closed residue field. Using Theorem~\ref{theorem:intro-jordan} and a simple argument in the case that $X$ is fiberwise semisimple, we reduce to the case that $X$ is fiberwise nilpotent. Next, using the theory of associated cocharacters we build an $A$-subgroup scheme $P$ of $G$ and an $S$-smooth closed subscheme $W$ of $P$ such that
\begin{enumerate}
    \item $g \in P(S)$,
    \item $Z_G(g) = Z_P(g)$ fibrally,
    \item the orbit map $\Phi: P \to P$, $\Phi(p) = pgp^{-1}$, factors through $W$, and
    \item the induced map $\Phi: P \to W$ has open dense image on fibers.
\end{enumerate}
If one can show that $Z_P(g)$ is flat, then the fibral isomorphism theorem and the second point imply together that the natural map $Z_P(g) \to Z_G(g)$ is an isomorphism, and in particular $Z_G(g)$ is flat. To show that $Z_P(g)$ is flat, we consider the Cartesian diagram
\[
\begin{tikzcd}
Z_P(g) \arrow[r] \arrow[d]
    &P \arrow[d, "\Phi"] \\
S \arrow[r, "g"]
    &W
\end{tikzcd}
\]
and we see that it suffices to show that $\Phi$ is flat. Using smoothness of $P$ and $W$, the fibral flatness criterion reduces us to the case that $S$ is the spectrum of a field, and ``Miracle Flatness" \cite[Thm.\ 23.1]{Matsumura} reduces one to showing that all nonempty fibers of $\Phi$ are of the same dimension, which is rather simple. The argument for sections of the Lie algebra is very similar.

\subsection{Notation and conventions}

We maintain all of the conventions specified in \cite[Sec.\ 1.2]{Integral-Springer}. Most notably, we will often conflate a locally free module over a ring with its associated affine space; if $S$ is a local scheme, then $s$ is always its closed point; if $S$ is an irreducible scheme then $\eta$ is always its generic point; and if $S$ is a scheme and $s \in S$, then $\overline{s}$ is always a geometric point lying over $s$. In \cite{SGA3III}, a \textit{reductive group scheme} is a smooth affine group scheme with \emph{connected} reductive fibers. We will generally follow this convention, but if $k$ is a field then a \emph{reductive group} over $k$ will not be assumed to be connected. If $G$ is a reductive group scheme then we denote its derived group by $\sD(G)$ and its center by $Z(G)$.

\subsection{Acknowledgements}

I thank Jeremy Booher for helpful correspondence, and for asking a question which led to Theorem~\ref{theorem:grothendieck-serre}. I thank Jay Taylor for pointing out an error in Section~\ref{subsection:pretty-good}. I thank Ravi Vakil for helpful conversations, and for suggesting that the material on Springer isomorphisms be extended into a separate paper. I thank an anonymous referee for a detailed reading and helpful comments. I thank my advisor, Brian Conrad, for very extensive edits and helpful suggestions on previous drafts.

\section{Some results over fields}

In this section we collect a number of results on pretty good primes and instability parabolics. The main point is that we are exercising care in proving schematic statements rather than simply statements on the level of varieties. In particular, we try not to assume (except where necessary) that the centers of our reductive groups are smooth.

\subsection{Pretty good primes}\label{subsection:pretty-good}

Recall \cite[I, \S 4.4]{Springer-Steinberg} that if $R = (X, \Phi)$ is a root system, then a prime number $p$ is \textit{bad} for $R$ if $\bZ\Phi/\bZ\Sigma$ has $p$-torsion for some closed subsystem $\Sigma \subset \Phi$, and it is \textit{good} otherwise. Similarly, a prime number $p$ is \textit{torsion} for $R$ if $\bZ\Phi^\vee/\bZ\Sigma^\vee$ has $p$-torsion for some closed subsystem $\Sigma \subset \Phi$. The following variant of these notions comes from \cite{Herpel}.

\begin{definition}
Let $\sR = (X, \Phi, Y, \Phi^\vee)$ be a root datum \cite[Exp.\ XXI, D\'ef.\ 1.1.1]{SGA3III} and let $p$ be a prime number. The prime $p$ is \textit{pretty good} for $\sR$ if, for all closed subsystems $\Phi' \subset \Phi$, the groups $X/\bZ\Phi'$ and $Y/\bZ\Phi'^\vee$ both have no $p$-torsion. If $G$ is a connected reductive group over a field $k$ of characteristic $p \geq 0$, then we will say that $p$ is \textit{pretty good} for $G$ if either $p=0$ or $p$ is pretty good for the root datum associated to $(G_{\overline{k}}, T)$ for some (any) maximal $\overline{k}$-torus $T$ of $G_{\overline{k}}$.
\end{definition}

We emphasize that unlike the conditions of being good or non-torsion, this is a property of the \textit{root datum} rather than the root system. In particular, when applied to a reductive group $G$ it takes into account both the isomorphism class (not just the isogeny class) of $\sD(G)_{\overline{k}}$ and the center of $G_{\overline{k}}$. The following lemma is the main description of pretty good primes that we will use.

\begin{lemma}\label{lemma:characterization-of-pretty-good-primes}
Let $G$ be a connected reductive group over a field $k$ of characteristic $p > 0$. Then $p$ is pretty good for $G$ if and only if the following conditions hold:
\begin{enumerate}
    \item $p$ is good for $G$,
    \item $p \nmid |\pi_1(\sD(G))|$, and
    \item $Z(G)$ is smooth.
\end{enumerate}
\end{lemma}

\begin{proof}
This is essentially a restatement of \cite[Lem.\ 2.12(a)]{Herpel} in the language of reductive groups. We may and do assume that $k$ is algebraically closed. Let $(X(T), \Phi, X_*(T), \Phi^\vee)$ be the root datum associated to a pair $(G, T)$ for some maximal $k$-torus $T$ of $G$. By \textit{loc.\ cit.}, $p$ is pretty good for $G$ if and only if $p$ is good for $G$ and $X(T)/\bZ\Phi$ and $X_*(T)/\bZ\Phi^\vee$ both have no $p$-torsion. Recall that $\pi_1(\sD(G))$ is Cartier dual to $((X_*(T)/\bZ\Phi^\vee)_{\rm{tors}})^*$ (where the asterisk denotes the dual finite abelian group) and $Z(G)$ is Cartier dual to $X(T)/\bZ\Phi$. The dual of a finitely generated constant commutative $k$-group scheme $A$ is smooth if and only if the order of $A_{\rm{tors}}$ is not divisible by $p$, whence the statement.
\end{proof}

\begin{lemma}\label{lemma:semisimple-replacement}
Let $G$ be a connected reductive group over an algebraically closed field $k$ of characteristic $p \geq 0$ such that $p$ is pretty good for $G$. If $X \in \fg$ is a semisimple element, then there exists some semisimple element $t \in G(k)$ such that $Z_G(t)^0 = Z_G(X)$.
\end{lemma}

\begin{proof}
Let $T$ be a maximal torus of $G$ containing $X$ in its Lie algebra. By \cite[Prop.\ 3.7]{Integral-Springer}, $Z_G(X)$ is a connected reductive group, and by \cite[Lem.\ 3.7]{SteinbergTorsion} it is generated by $T$ and those root groups $U_{\alpha}$ such that $\mathrm{d}\alpha(X) = 0$. Let $\Phi' \subset \Phi$ be the closed subsystem consisting of such $\alpha$, so by hypothesis we have
\[
\left(\bigcap_{\alpha \in \Phi'} \ker(\mathrm{d}\alpha)\right) - \left(\bigcup_{\beta \in \Phi - \Phi'} \ker(\mathrm{d}\beta)\right) \neq \emptyset.
\]
By the above displayed equation, we see that for any $\beta \in \Phi - \Phi'$ we have in particular
\[
\dim\left(\bigcap_{\alpha \in \Phi' \cup \{\beta\}} \ker(\mathrm{d}\alpha)\right) < \dim\left(\bigcap_{\alpha \in \Phi'} \ker(\mathrm{d}\alpha)\right).
\]
Since $p$ is pretty good for $G$, the intersection $\bigcap_{\alpha \in \Phi'} \ker(\alpha)$ (which is Cartier dual to $X(T)/\bZ\Phi'$) is a smooth closed subgroup of $T$, and it follows from the previous displayed equation and smoothness that
\[
\dim\left(\bigcap_{\alpha \in \Phi' \cup \{\beta\}} \ker(\alpha)\right) < \dim\left(\bigcap_{\alpha \in \Phi'} \ker(\alpha)\right),
\]
so for dimension reasons it follows that
\[
\left(\bigcap_{\alpha \in \Phi'} \ker(\alpha)\right) - \left(\bigcup_{\beta \in \Phi - \Phi'} \ker(\beta)\right) \neq \emptyset.
\]
If $t$ is any closed point of this set, then the descriptions of semisimple centralizers in \cite[Lems.\ 2.14 and 3.7]{SteinbergTorsion} show that $Z_G(t)^0 = Z_G(X)^0$, and thus $Z_G(t)^0 = Z_G(X)$ since $Z_G(X)$ is connected.
\end{proof}

The following remarkable theorem is the main result of \cite{Herpel}.

\begin{theorem}\label{theorem:herpels-theorem}
Let $G$ be a connected reductive group over a field $k$. Then the characteristic of $k$ is pretty good for $G$ if and only if all centralizers of closed subgroup schemes in $G$ are smooth.
\end{theorem}

\begin{cor}\label{corollary:smooth-centralizer-pretty-good-characteristic}
Let $G$ be a connected reductive group over a field $k$ of characteristic $p > 0$, and let $g \in G(k)$ (resp.\ $X \in \Lie G$) be any element. Suppose that $p$ is good for $G$ and $p \nmid |\pi_1(\sD(G))|$. Then $Z_G(g)/Z(G)$ (resp.\ $Z_G(X)/Z(G)$) is smooth.
\end{cor}

\begin{proof}
We may and do assume that $k$ is algebraically closed. The case $p=0$ is immediate from Cartier's theorem, so we may and do assume $p > 0$. Choose a torus $T$ and an embedding $Z(G) \to T$. Let $G' = G \times^{Z(G)} T$, so $Z(G') = T$, $G$ and $G'$ have the same derived group and hence the same root system, there is a natural embedding of $G$ in $G'$, and we have $Z_{G'}(g) = Z_G(g) \times^{Z(G)} T$ (and similarly for $Z_{G'}(X)$). This implies that $Z_{G'}(g)/Z(G') = Z_G(g)/Z(G)$ (and similarly for $Z_{G'}(X)/Z(G')$). Since passing from $G$ to $G'$ does not affect the derived group, we may and do assume that $Z(G)$ is smooth, in which case $p$ is pretty good for $G$ by Lemma~\ref{lemma:characterization-of-pretty-good-primes}. \smallskip

Now for the first claim, let $H = \overline{\langle g \rangle}$, so that $H$ is a closed subgroup scheme of $G$. Moreover, $Z_G(g) = Z_G(H)$, so the result follows from Theorem~\ref{theorem:herpels-theorem}. \smallskip

For the second claim, let $X = X_{\rm{ss}} + X_{\rm{n}}$ be the Jordan decomposition of $X$. By \cite[Lem.\ 2.6]{Integral-Springer}, we have $Z_G(X) = Z_{Z_G(X_{\rm{ss}})}(X_{\rm{n}})$. By \cite[Prop.\ 3.7]{Integral-Springer}, $Z_G(X_{\rm{ss}})$ is a connected reductive group, and clearly $p$ is also good for $Z_G(X_{\rm{ss}})$. To see that $p$ is even pretty good for $Z_G(X_{\rm{ss}})$, we note that by Lemma~\ref{lemma:semisimple-replacement} there is some semisimple element $t \in G(k)$ such that $Z_G(t)^0 = Z_G(X_{\rm{ss}})$. Thus the claim follows from Lemma~\ref{lemma:characterization-of-pretty-good-primes} and \cite[Cor.\ 3.2]{Integral-Springer}, and we may pass from $G$ to $Z_G(X_{\rm{ss}})$ and from $X$ to $X_{\rm{n}}$ to assume that $X$ is nilpotent. Since $p$ is good for $G$, the existence of a Springer isomorphism \cite[Thm.\ 5.1]{Integral-Springer} shows that there is some unipotent element $u \in G(k)$ such that $Z_G(u) = Z_G(X)$. Thus $Z_G(X)/Z(G)$ is smooth by the first part.
\end{proof}

\begin{remark}
The assumption that the characteristic of a field is good for a given reductive group is a ``standardness" assumption, of which there are many competing notions in the literature. These notions are very closely related, as \cite[Thm.\ 5.2]{Herpel} shows. In the rest of this paper, we may cite references which invoke different standardness assumptions, but we will generally ignore this because \cite[Thm.\ 5.2]{Herpel} makes it rather easy to transfer results involving one standardness hypothesis to results using the assumption of pretty good characteristic.
\end{remark}

\subsection{(Almost) associated cocharacters and instability parabolics}\label{subsection:associated-cocharacters}

In this section, we recall some definitions and results from \cite[\S 5]{Jantzen-nilpotent} and \cite{McNinch-nilpotent-orbits}. For case-free proofs using geometric invariant theory, see \cite{Premet}. Throughout this section, $G$ is a connected reductive group over a field $k$ of characteristic $p \geq 0$, and $\fg = \Lie G$.

\begin{definition}
\begin{enumerate}
    \item We say that a nilpotent element $X \in \fg$ is \textit{distinguished} if the only maximal torus in $(Z_G(X)_{\overline{k}})_{\rm{red}}$ is the maximal central torus of $G_{\overline{k}}$. Note that $X = 0$ is distinguished if and only if $G$ is a torus.
    \item We say that a parabolic subgroup $P \subset G$ is \textit{distinguished} if 
    \[
    \dim P_{\overline{k}}/\sR_u(P_{\overline{k}}) = \dim \sR_u(P_{\overline{k}})/\sR_u(P_{\overline{k}})^1 + \dim Z(G),
    \]
    where $\sR_u(P_{\overline{k}})^1$ is the subgroup of $\sR_u(P)$ generated by those root groups $U_\alpha$ (with respect to a fixed maximal torus of $P_{\overline{k}}$) such that $\alpha$ is not a sum of two roots in $\Phi(\sR_u(P_{\overline{k}}))$. Note that $P = G$ is distinguished if and only if $G$ is a torus.
\end{enumerate}
\end{definition}

For example, any Borel subgroup is distinguished. In type $\rm{A}$ all distinguished parabolic subgroups are of this form, but this case is misleading; in every other type, there exist distinguished parabolic subgroups which are not Borel. See \cite[\S 5.9]{Carter} for details. \smallskip

We recall that if $P$ is a parabolic subgroup of a reductive group $G$ over an infinite field, then there is a unique open dense $\Ad_P$-orbit in $\Lie \sR_u(P)$, and it is called the \textit{Richardson orbit} of $P$.\footnote{Recall that the $\Ad_P$-orbit of an element $X \in \Lie \sR_u(P)$ is by definition the schematic image $\cO_X$ of the morphism $P \to \Lie \sR_u(P)$ given by $p \mapsto \Ad(p)X$. Since the formation of the schematic image commutes with flat base change, it follows by passing to $\ov k$ that $\cO_X$ is geometrically reduced and thus (by generic flatness and a translation argument) the orbit map $P \to \cO_X$ is flat and hence a $Z_P(X)$-torsor. These observations typically allow one to pass from $k$ to $\ov k$, as we will do below.} By \cite[4.13]{Jantzen-nilpotent}, if $P$ is a distinguished parabolic in good characteristic then the Richardson orbit of $P$ consists of distinguished nilpotent elements. Moreover, every distinguished nilpotent element lies in the Richardson orbit of some distinguished parabolic subgroup. These facts are key to the following theorem, known as the Bala--Carter--Pommerening theorem.

\begin{theorem}[{\cite[4.7, 4.13]{Jantzen-nilpotent}}]\label{theorem:bala-carter}
Suppose that $\chara k$ is good for $G$ and $k$ is algebraically closed. Consider the collection of pairs $(L, Q)$, where $L$ is a Levi subgroup of some parabolic of $G$ and $Q$ is a distinguished parabolic in $L$. Associate to each such pair the $\Ad_G$-orbit of the Richardson orbit of $Q$. This map induces a bijection between $G$-conjugacy classes of pairs $(L, Q)$ and nilpotent $\Ad_G$-orbits in $\fg$.
\end{theorem}

Let $X \in \fg$ be nilpotent and let $(L, Q)$ be a pair consisting of some Levi subgroup $L$ of a parabolic of $G$ and a distinguished parabolic $Q$ of $L$. If $X \in \Lie Q$ is an element of the Richardson orbit, then we will call $(L, Q)$ a \textit{Bala--Carter datum} of $X$ (noting that it is unique up to geometric conjugacy). For $X = 0$, we must have $L = Q$, and each must be a maximal torus of $G$.

\begin{definition}\label{definition:associated-cocharacters}
Let $X \in \fg$ be a nilpotent element. We say that a cocharacter $\tau: \bG_m \to G$ is \textit{associated to} $X$ if
\begin{itemize}
    \item $\Ad(\tau(t))X = t^2 X$ for all $k$-algebras $R$ and all $t \in \bG_m(R)$, and
    \item there exists a Levi $k$-subgroup $L$ of $G$ such that $\tau$ factors through $\sD(L)$ and $X$ is distinguished in $\Lie L$.
\end{itemize}
\end{definition}

Note that if $X = 0$, then $L$ must be a torus, so the second condition implies $\tau = 1$. 

\begin{example}
Let $G = \SL_3$ and consider the element
\[
X = \begin{pmatrix} 0 & 1 & 0 \\ 0 & 0 & 0 \\ 0 & 0 & 0 \end{pmatrix}.
\]
Consider the two cocharacters $\tau_1, \tau_2: \bG_m \to G$ given by
\begin{gather*}
\tau_1(t) = \begin{pmatrix} t & 0 & 0 \\ 0 & t^{-1} & 0 \\ 0 & 0 & 1 \end{pmatrix} \\
\tau_2(t) = \begin{pmatrix} t^2 & 0 & 0 \\ 0 & 1 & 0 \\ 0 & 0 & 1 \end{pmatrix}
\end{gather*}
Note that $X$ lies in the $2$-weight space of both $\tau_1$ and $\tau_2$, but only $\tau_1$ is associated to $X$. Namely, note that $\tau_1$ factors through the derived group of the Levi
\[
L = \begin{pmatrix} * & * & 0 \\ * & * & 0 \\ 0 & 0 & * \end{pmatrix}.
\]
While $\tau_2$ also factors through $L$, it does not factor through $\sD(L)$. Although this does not evidently rule out the possibility that $\tau_2$ is associated to $X$, we will see in the next lemma that any two associated cocharacters of $X$ are conjugate, and it is clear that $\tau_1$ and $\tau_2$ are not conjugate. In general, associated cocharacters are a sort of group variant of the notion of $\mathfrak{sl}_2$-triple.
\end{example}

\begin{lemma}\label{lemma:existence-of-associated-cocharacters}\cite[5.3]{Jantzen-nilpotent}, \cite[Thm.\ 26]{McNinch-nilpotent-orbits}
Let $X \in \fg$ be a nilpotent element.
\begin{enumerate}
    \item\label{item:associated-cocharacters-exist} If $p$ is good for $G$ and either $k$ is perfect or $p \nmid |\pi_1(\sD(G))|$, then cocharacters associated to $X$ exist over $k$. 
    \item\label{item:associated-cocharacters-bala-carter} If $k$ is algebraically closed and $(L, Q)$ is a Bala--Carter datum of $X$, then there exists a cocharacter $\tau: \bG_m \to \sD(L)$ associated to $X$ such that $Q = P_L(\tau)$.
    \item\label{item:associated-cocharacters-conjugate} Two cocharacters associated to $X$ are conjugate under $Z_G(X)^0(\overline{k})$.
\end{enumerate}
\end{lemma}

\begin{proof}
This is proved in \cite[5.3]{Jantzen-nilpotent} under the further hypothesis that $k$ is algebraically closed. Note that the second claim in the first point is not stated in \cite[5.3]{Jantzen-nilpotent}, but the construction of associated cocharacters establishes it. The first point is proved in \cite[Thm.\ 26]{McNinch-nilpotent-orbits} under the slightly stronger assumption that either $k$ is perfect or $Z_G(X)$ is smooth. Thus it suffices to understand the case $p \nmid |\pi_1(\sD(G))|$. In that case, choose an embedding $Z(G) \to T_0$, where $T_0$ is a torus, and let $G' = G \times^{Z(G)} T_0$. Note $Z(G') = T_0$, so by \cite[Thm.\ 26]{McNinch-nilpotent-orbits}, there is a cocharacter $\tau': \bG_m \to G'$ associated to $X$. Let $T'$ be a maximal $k$-torus of $G'$ through which $\tau'$ factors, and write $T' = TT_0$, where $T$ is a maximal $k$-torus of $G$. We may write $\tau' = \tau \tau_0$ for some (non-unique) cocharacters $\tau: \bG_m \to T$ and $\tau_0: \bG_m \to T_0$, and it is clear that $\tau$ is associated to $X$. Since $\tau$ factors through $G$, we conclude.
\end{proof}

\begin{remark}
The hypotheses in Lemma~\ref{lemma:existence-of-associated-cocharacters}(\ref{item:associated-cocharacters-exist}) are necessary: if $D$ is a central division algebra of dimension $p^2$ over a local function field $\bF_q(\!(t)\!)$ of characteristic $p > 0$, then $\pgl(D)$ contains nontrivial nilpotent elements, corresponding to elements of $\ov{\bF_q(\!(t)\!)}$ which are purely inseparable of degree $p$ over $\bF_q(\!(t)\!)$. However, $\PGL(D)$ is anisotropic, so it admits no nontrivial cocharacters over $k$.
\end{remark}

Given a nilpotent element $X \in \fg$ and a cocharacter $\tau: \bG_m \to G$ associated to $X$, we let $P = P_G(\tau)$ be the parabolic subgroup of $G$ associated to $\tau$ by the dynamic method; for details on the dynamic method, see \cite[\S 2.1]{CGP}. Let also $L_\tau = Z_G(\tau)$ be the Levi factor of $P$ associated to $\tau$, and let $U_P = U_G(\tau)$ be the unipotent radical of $P$. We call $P$ the \textit{instability parabolic} associated to $X$; the word ``the" and the $\tau$-independent notation are justified by the first part of the following proposition. For a given cocharacter $\lambda: \bG_m \to G$, we denote by $\fg(\lambda, n)$ the space of weight $n$ vectors for $\lambda$ in $\fg$.

\begin{prop}[{\cite[5.9-11]{Jantzen-nilpotent}}]\label{proposition:properties-of-instability-parabolic}
Suppose that $p$ is good for $G$.
\begin{enumerate}
    \item\label{item:instability-parabolic-is-choiceless} The group $P$ depends only on $X$, not on the choice of $\tau$.
    \item\label{item:instability-parabolic-same-centralizer} If $p \nmid |\pi_1(\sD(G))|$, then $Z_G(X) = Z_P(X)$.
    \item\label{item:instability-parabolic-open-dense-orbit} The $\Ad_P$-orbit of $X$ is open and dense in $\bigoplus_{n \geq 2} \fg(\tau, n)$.
    \item\label{item:instability-parabolic-bala-carter} If $\tau$ is an associated cocharacter of $X$ and $L$ is as in the definition, then $(L, P_L(\tau))$ is a Bala--Carter datum of $X$.
    \item\label{item:instability-parabolic-semidirect-product} Suppose $p \nmid |\pi_1(\sD(G))|$. If $\tau$ is an associated cocharacter of $X$, then
    \[
    Z_G(X) = (Z_G(X) \cap Z_G(\tau)) \ltimes Z_{U_G(\tau)}(X).
    \]
    The group $Z_{U_G(\tau)}(X)$ is smooth, connected, and unipotent; the quotient group scheme $(Z_G(X) \cap Z_G(\tau))/Z(G)$ is reductive (but not necessarily connected).
    \item\label{item:instability-parabolic-levi} Suppose $p \nmid |\pi_1(\sD(G))|$. If $\tau$ is an associated cocharacter of $X$, $T$ is a maximal $k$-torus of $(Z_G(X) \cap Z_G(\tau))/Z(G)$, and $L = Z_G(T)$, then $\tau$ factors through $\sD(L)$ and $X$ is distinguished in $\Lie L$. In particular, $(L, P_L(\tau))$ is a Bala--Carter datum of $X$.
\end{enumerate}
\end{prop}

\begin{proof}
We may and do assume that $k$ is algebraically closed. (\ref{item:instability-parabolic-is-choiceless}) is proved in the reference. For (\ref{item:instability-parabolic-same-centralizer}), since $Z(G) \subset Z_P(X)$, we may check that the morphism $Z_P(X) \to Z_G(X)$ is an isomorphism after passing to the quotient by $Z(G)$. By Corollary~\ref{corollary:smooth-centralizer-pretty-good-characteristic} the group $Z_G(X)/Z(G)$ is smooth, so in fact it suffices to check that $Z_P(X)/Z(G) \to Z_G(X)/Z(G)$ is bijective on $k$-points. This is proved in \cite[5.9]{Jantzen-nilpotent}. The density statement of (\ref{item:instability-parabolic-open-dense-orbit}) is proved in \cite[5.9]{Jantzen-nilpotent}, while openness follows from the closed orbit lemma \cite[1.8, Prop.]{Borel}. For (\ref{item:instability-parabolic-bala-carter}) we note that there exists \textit{some} cocharacter $\tau: \bG_m \to G$ associated to $X$ satisfying the conclusion by Lemma~\ref{lemma:existence-of-associated-cocharacters}(\ref{item:associated-cocharacters-exist}). By Lemma~\ref{lemma:existence-of-associated-cocharacters}(\ref{item:associated-cocharacters-conjugate}), any two associated cocharacters are $Z_G(X)^0$-conjugate, so the claim follows. \smallskip

For (\ref{item:instability-parabolic-semidirect-product}) we may embed $Z(G)$ into a torus $T$ and pass from $G$ to $G \times^{Z(G)} T$ to assume that $Z(G)$ is smooth, in which case $Z_G(X)$ is smooth by Corollary~\ref{corollary:smooth-centralizer-pretty-good-characteristic}. The fact that the multiplication morphism $(Z_G(X) \cap Z_G(\tau)) \ltimes Z_{U_G(\tau)}(X) \to Z_G(X)$ is an isomorphism of group schemes follows from \cite[5.10]{Jantzen-nilpotent}: there it is proved that this is an isomorphism on the level of geometric points and on the level of Lie algebras. For dimension reasons, it follows that $Z_G(X) \cap Z_G(\tau)$ and $Z_{U_G(\tau)}$ are smooth, and thus indeed the multiplication morphism is an isomorphism. The remaining claims of (\ref{item:instability-parabolic-semidirect-product}) follow from \cite[5.10-11]{Jantzen-nilpotent}. \smallskip

Finally, for (\ref{item:instability-parabolic-levi}) we may reduce as before to the case that $Z(G)$ is smooth, in which case the maximal torus $T$ corresponds to a maximal torus of $Z_G(X) \cap Z_G(\tau)$. Using this correspondence, the result is proved in \cite[Prop.\ 16]{McNinch-nilpotent-orbits}.
\end{proof}

If $X \in \fg$ is nilpotent, we let $N_G(X)$ denote the normalizer of $kX$ in $\fg$. In other words, $N_G(X)$ is the closed $k$-subgroup scheme of $G$ fitting into the Cartesian square
\begin{equation}\label{equation:normalizer-scheme}
\begin{tikzcd}
N_G(X) \arrow[r] \arrow[d]
    &kX \arrow[d] \\
G \arrow[r, "\Phi"]
    &\fg
\end{tikzcd}
\end{equation}
where $\Phi: G \to \fg$ is the orbit map $\Phi(g) = \Ad(g)X$.

\begin{lemma}\label{lemma:normalizer-scheme}
There is a short exact sequence $1 \to Z_G(X) \to N_G(X) \to \bG_m \to 1$, where $N_G(X) \to \bG_m$ is given by sending $g$ to the coefficient of $X$ in $\Ad(g)X$. If $\tau: \bG_m \to G$ is associated to $X$, then $\tau$ factors through $N_G(X)$ and the composition $\bG_m \to N_G(X) \to \bG_m$ is given by squaring. If $p$ is good for $G$ and $p \nmid |\pi_1(\sD(G))|$, then $N_G(X)/Z(G)$ is smooth.
\end{lemma}

\begin{proof}
For the first statement, the only unclear point is that $N_G(X) \to \bG_m$ is faithfully flat. For this, note that if $\tau: \bG_m \to G$ is an associated cocharacter of $X$, then $\tau$ factors through $N_G(X)$ and the composition $\bG_m \to N_G(X) \to \bG_m$ is given by squaring (as is easy to see). Since the squaring morphism is an epimorphism of fppf sheaves, the claim follows. The last statement follows from Corollary~\ref{corollary:smooth-centralizer-pretty-good-characteristic} using the proven short exact sequence.
\end{proof}

Finally, if $X \in \fg$ is nilpotent and $P$ is the instability parabolic associated to $X$, then as in \cite[4.3]{McNinch-relative-centralizer} we will say that a cocharacter $\tau: \bG_m \to G$ is \textit{almost associated} to $X$ provided that $t \mapsto g\tau(t)g^{-1}$ is associated to $X$ for some $g \in P(\overline{k})$.\footnote{A referee pointed out that if $k$ is perfect then it is equivalent to assume that $\lambda\colon t \mapsto g\tau(t)g^{-1}$ is associated to $X$ for some $g \in P(k)$. Indeed, in the definition we may assume $g \in \sR_u(P)(\ov k)$. In this case, we have $gZ_G(\tau)g^{-1} = Z_G(\lambda)$, so such a $g$ is unique and therefore Galois-stable.} In particular, $\tau$ necessarily factors through $P$. The principal properties of such $\tau$ which we will need are contained in the following proposition.

\begin{prop}\label{proposition:mcninch-almost-associated-cocharacters}
Let $X \in \fg$ be nilpotent and let $P$ be the instability parabolic associated to $X$.
\begin{enumerate}
    \item\label{item:almost-associated-cocharacters-exist}\cite[4.3.1]{McNinch-relative-centralizer} Let $T \subset P$ be a maximal torus. Then there is a unique cocharacter $\tau$ of $T$ which is almost associated to $X$.
    \item\label{item:almost-associated-cocharacters-dynamic-method} If $\tau$ is a cocharacter which is almost associated to $X$, then $P_G(\tau)$ is the instability parabolic of $X$.
    \item\label{item:almost-associated-cocharacters-agree}\cite[4.3.2]{McNinch-relative-centralizer} If $\tau_1, \tau_2$ are two cocharacters which are almost associated to $X$, then for any $N \in \bZ$ we have
    \[
    \bigoplus_{n \geq N} \fg(\tau_1, n) = \bigoplus_{n \geq N} \fg(\tau_2, n).
    \]
\end{enumerate}
\end{prop}

\begin{proof}
First, (\ref{item:almost-associated-cocharacters-exist}) is proved as stated in the reference. For (\ref{item:almost-associated-cocharacters-dynamic-method}), note that $\tau$ is $P(\overline{k})$-conjugate to an associated cocharacter $\tau_0$ of $X$, and thus $P_G(\tau) = P_G(\tau_0)$ is the instability parabolic of $X$. \smallskip

For (\ref{item:almost-associated-cocharacters-agree}), note that each side of the displayed equation is $\Ad_P$-stable: $Y$ lies in the left side, for example, if and only if the limit $\lim_{t \to 0} t^{-N} \Ad(\tau_1(t))Y$ exists. If $g \in P(\overline{k})$, then we have
\[
\lim_{t \to 0} t^{-N} \Ad(\tau_1(t)g)Y = \lim_{t \to 0} \Ad(\tau_1(t)g\tau_1(t)^{-1}) (t^{-N}\Ad(\tau_1(t))Y).
\]
By definition of $P_G(\tau_1)$, the limit $\lim_{t \to 0} \tau_1(t)g\tau_1(t)^{-1}$ exists, so it follows that the above limit exists whenever $\lim_{t \to 0} t^{-N} \Ad(\tau_1(t))Y$ exists. This establishes the desired $\Ad_P$-stability, and the equality follows from the fact that $\tau_1$ and $\tau_2$ are $P(\overline{k})$-conjugate by hypothesis.
\end{proof}

\begin{remark}
In Definition~\ref{definition:associated-cocharacters}, it is immaterial that the Levi $L$ is required to be defined over $k$. Indeed, let $\tau: \bG_m \to G$ be a cocharacter such that $\tau_{\overline{k}}$ is associated to $X_{\overline{k}}$. If $T$ is a maximal $k$-torus of $(Z_G(X) \cap Z_G(\tau))/Z(G)$, then by Proposition~\ref{proposition:properties-of-instability-parabolic}(\ref{item:instability-parabolic-levi}) we see that if $L = Z_G(T)$ then $L$ is a Levi $k$-subgroup of $G$ satisfying the conditions of Definition~\ref{definition:associated-cocharacters}.
\end{remark}

\section{Pure and fiberwise ordinary sections}\label{section:constructibility-nonsense}

In this section we define purity and strong purity, and we aim to analyze various loci related to these. The main aim is to prove constructibility and openness of various subsets of $G$ and $\fg$. \smallskip

Throughout, $G$ is a reductive group scheme over a base scheme $S$. For simplicity we will assume that $S$ is quasi-compact and quasi-separated, although all statements below (suitably altered) remain true without this assumption. For definitions and general results on constructibility, see \cite[0\textsubscript{III}, \S 9; IV\textsubscript{1}, \S 1.8-9]{EGA}. We begin with the simplest of our constructibility results.

\begin{lemma}\label{lemma:semisimple-locus-is-open}
The subset
\[
G_{\rm{ss}} = \{g \in G: g \text{ is semisimple in } G(\overline{k(g)})\}
\]
is a fiberwise dense constructible subset of $G$. The analogous subset $\fg_{\rm{ss}}$ is a fiberwise dense constructible subset of $\fg$.
\end{lemma}

\begin{proof}
Working etale-locally on $S$, we may assume that there exists a maximal torus $T$ of $G$. Over an algebraically closed field, any semisimple element is conjugate to an element of any given maximal torus. Thus if $\mu_T: G \times T \to G$ is defined by $\mu_T(g, t) = gtg^{-1}$ then we have $G_{\rm{ss}} = \mu_T(G \times T)$, which is constructible by Chevalley's theorem \cite[IV\textsubscript{1}, Thm.\ 1.8.4]{EGA}. Fiberwise density of $G_{\rm{ss}}$ reduces to a statement about fields, which is well-known. The proof for $\fg_{\rm{ss}}$ is completely analogous, using the fact that semisimple elements of $\fg$ are all $\Ad_G$-conjugate to an element of $\Lie T$ over an algebraically closed field.
\end{proof}

\subsection{Ordinary elements}

We begin in earnest our study of constructibility properties by introducing the notion of an ordinary element and showing that under favorable hypotheses, the locus of ordinary elements is open.

\begin{definition}\label{definition:ordinary}
Let $S = \Spec k$ be the spectrum of a field. An element $g \in G(k)$ is \textit{ordinary} provided that, in the Jordan decomposition $g_{\overline{k}} = tu$, the centralizer $Z_{G_{\overline{k}}}(t)$ is connected. For general $S$, let $G_{\rm{ord}}$ denote the set of $g \in G$ such that $g$ is ordinary in $G(k(g))$.
\end{definition}

Note that if $\sD(G)$ is simply connected, then every element of $G(k)$ is ordinary by \cite[Theorem 2.15]{SteinbergTorsion}. We will see that $G_{\rm{ord}}$ is always fiberwise dense and constructible in $G$, and it is open under some weak hypotheses on $G$. We begin by giving an example to show that $G_{\rm{ord}}$ is not always open.

\begin{example}
Consider the section
\[
g = \begin{pmatrix} 1 & 1 \\ 0 & -1 \end{pmatrix} \in \PGL_2(\bZ_2).
\]
The special fiber of $g$ is unipotent, while the generic fiber is semisimple but not ordinary; the centralizer has two connected components.
\end{example}

Ultimately, this example motivates the hypotheses of the upcoming Proposition~\ref{prop:strong-purity-generizes}. We begin, in any case, by establishing constructibility.

\begin{lemma}\label{lemma:ordinary-locus-is-locally-constructible}
$G_{\rm{ord}}$ is a fiberwise dense constructible subset of $G$.
\end{lemma}

\begin{proof}
Since etale morphisms are open, we may pass to an etale cover of $S$ to assume that $G$ is $S$-split. Since $G_{\rm{ord}}$ is stable under $Z(G)$-translation, we may further assume that $G$ is semisimple. Fix a split maximal torus $T$ of $G$ and consider the subset
\[
T_{\rm{ord}} = \{t \in T: t \text{ is ordinary in } T(k(t))\}
\]
of $T$. We will show first that $T_{\rm{ord}}$ is a fiberwise dense constructible subset of $T$. \smallskip

By \cite[Lem.\ 2.14]{SteinbergTorsion}, if $t \in T(k)$ for some field $k$, then the connected centralizer $Z_{G_k}(t)^0$ is generated by $T$ and those root groups $U_\alpha$ such that $\alpha(t) = 1$. Moreover, the component group $Z_{G_k}(t)/Z_{G_k}(t)^0$ is generated by representatives $n_w$ of those elements $w$ of the Weyl group $W$ such that $wtw^{-1} = t$. For each closed subsystem $\Phi' \subset \Phi$, let $T_{\Phi'}$ be the set of those $t \in T$ such that $\alpha(t) = 1$ if and only if $\alpha \in \Phi'$; each subset $T_{\Phi'}$ is evidently a locally closed subset of $T$, and these subsets form a stratification of $T$. To show that $T_{\rm{ord}}$ is constructible, it suffices to show that each subset $T_{\Phi', \rm{ord}} = T_{\Phi'} \cap T_{\rm{ord}}$ is open in $T_{\Phi'}$. But this is clear, since the above description of $Z_G(t)/Z_G(t)^0$ shows that
\[
T_{\Phi', \rm{ord}} = \{t \in T_{\Phi'}: wtw^{-1} \neq t \text{ for all } w \in W - W(\Phi')\}.
\]

Now fix a Borel $S$-subgroup $B$ of $G$ containing $T$, and let $U$ be the unipotent radical of $B$. We aim to show that $B \cap G_{\rm{ord}}$ is a constructible subset of $B$. More precisely, we claim that
\[
B \cap G_{\rm{ord}} = T_{\rm{ord}} \times U.
\]
A simple lemma of Steinberg \cite[Lem.\ 2.12]{SteinbergReg} shows that if $b \in B(k)$ for some algebraically closed field $k$ and $b = tu$ for $t \in T(k)$, $u \in U(k)$, then $b$ is $B(k)$-conjugate to an element of the form $b' = tu'$, where $u' \in U(k)$ commutes with $t$. The expression $b' = tu'$ is necessarily the Jordan decomposition of $b'$, so that $b'$ is ordinary if and only if $t$ is ordinary. But the property of being ordinary is invariant under conjugacy, so we see that indeed $b$ is ordinary if and only if $t$ is ordinary, proving the claim. In particular, $B \cap G_{\rm{ord}}$ is constructible in $G$. \smallskip

Define now the morphism $\mu_B: G \times B \to G$ via
\[
\mu_B(g, b) = gbg^{-1}.
\]
For every algebraically closed field $k$ and every $g \in G(k)$, there is some $G(k)$-conjugate of $g$ contained in $B(k)$, so we see that
\[
G_{\rm{ord}} = \mu_B(G \times B_{\rm{ord}}).
\]
In particular, $G_{\rm{ord}}$ is constructible by Chevalley's theorem. Fiberwise density follows from the fact that $G_{\rm{ord}}$ contains, for each geometric point $\overline{s}$ of $S$, the locus of strongly regular semisimple elements of $G(k(\overline{s}))$, which is dense and open in $G(k(\overline{s}))$ by \cite[2.15]{SteinbergReg}.
\end{proof}

To show that $G_{\rm{ord}}$ is open, we must first establish a few lemmas. We begin with a simple corollary of the ordinary Borel covering theorem.

\begin{lemma}\label{lemma:dvr-covering-by-borels}
If $S = \Spec A$ is a DVR and $g \in G(A)$ then there is some generically finite local extension of DVRs $A \to A'$ and a Borel $A'$-subgroup $B$ of $G_{A'}$ containing $g_{A'}$.
\end{lemma}

\begin{proof}
Let $K$ be the fraction field of $A$. By the usual theorem on coverings by Borel subgroups, there is a finite extension $K'$ of $K$ and a Borel subgroup $B'$ of $G_{K'}$ containing $g_{K'}$. Let $A'$ be the localization at a maximal ideal of the normalization of $A$ in $K'$, so that $A'$ is a DVR by the Krull-Akizuki theorem \cite[Thm.\ 11.7]{Matsumura}. Since the scheme of Borel subgroups is proper \cite[Cor.\ 5.2.8]{Conrad}, there is some Borel $A'$-subgroup $B$ of $G_{A'}$ such that $B_{K'} = B'$. Thus $B$ contains $g_{A'}$, as desired.
\end{proof}

\begin{lemma}\label{lemma:Weyl-elements-in-identity-component}
Suppose $S = \Spec k$ for an algebraically closed field $k$ and $G$ is connected and semisimple. Let $T$ be a maximal torus of $G$, let $W$ be the Weyl group of the pair $(G, T)$, and let $t \in T(k)$. Let $\pi: \widetilde{G} \to G$ be the universal cover, and let $\widetilde{t} \in \widetilde{G}(k)$ lift $t$. The induced morphism $Z_{\widetilde{G}}(\widetilde{t}) \to Z_G(t)^0$ is a central isogeny of connected reductive groups. Moreover, if $w \in W$ centralizes $t$, then $w$ has a representative in $Z_G(t)^0(k)$ if and only if $w$ centralizes $\widetilde{t}$.
\end{lemma}

\begin{proof}
By \cite[Thm.\ 2.15]{SteinbergTorsion}, $Z_{\widetilde{G}}(\widetilde{t})$ is connected, and the morphism $Z_{\widetilde{G}}(\widetilde{t}) \to Z_G(t)^0$ is clearly finite with central kernel, so it suffices to show that it is dominant. This can be seen easily by considering compatible open cells for $T$ (and $\pi^{-1}(T)$) in $G$ and $\widetilde{G}$. This proves the first claim. \smallskip

If $w \in W$ centralizes $\widetilde{t}$, then there is a representative $n_w$ of $w$ lying in the image of $Z_{\widetilde{G}}(\widetilde{t}) \to Z_G(t)$, so $n_w \in Z_G(t)^0(k)$. Conversely, suppose that $w$ has a representative $n_w \in Z_G(t)^0(k)$. Any lift of $n_w$ to $Z_{\widetilde{G}}(\widetilde{t})(k)$ is also a representative for $w$ and hence $w$ centralizes $\widetilde{t}$, as desired.
\end{proof}

\begin{prop}\label{prop:strong-purity-generizes}
Let $S = \Spec A$ for a DVR $A$, suppose either that $A$ is equicharacteristic or that $|\pi_1(\sD(G))|$ is invertible in $A$, and let $g \in G(A)$ be a section. If $g_s$ is ordinary, then also $g_\eta$ is ordinary. \smallskip

More generally, suppose either that locally on $S$ there exists a prime number $p$ which is nilpotent in $\Gamma(S, \sO_S)$ or that $|\pi_1(\sD(G))|$ is invertible on $S$. Then $G_{\rm{ord}}$ is open.
\end{prop}

\begin{proof}
After passing to a local extension of $A$, we may write $g = zh$ for $z \in Z(G)(A)$ and $h \in \sD(G)(A)$. Replacing $g$ by $h$ and $G$ by $\sD(G)$, we may thereby assume that $G$ is semisimple. By Lemma~\ref{lemma:dvr-covering-by-borels}, after further extending $A$ we may assume that there is a Borel $A$-subgroup $B$ of $G$ containing $g$. Further extending $A$, we may assume that $B$ contains a split maximal $A$-torus $T$. Let $g = tu$ for $t \in T(A)$ and $u \in U(A)$. For both points $x \in \Spec A$ the semisimple part of $g_{\overline{x}}$ is conjugate to $t_{\overline{x}}$ by \cite[Lem.\ 2.12]{SteinbergReg}, so it suffices to show that $t_\eta$ has connected centralizer (given that $t_s$ has connected centralizer). \smallskip

Let $W$ be the Weyl group of the pair $(G, T)$. By \cite[Lem.\ 2.14]{SteinbergTorsion}, it suffices to show that if $w \in W$ centralizes $t_\eta$, then $w$ has a representative lying in $Z_{G_\eta}(t_\eta)(k(\overline{\eta}))$. Of course, if $w$ centralizes $t_\eta$ then it also centralizes $t$, and hence $t_s$. Let now $\pi: \widetilde{G} \to G$ be the universal cover, and let $\widetilde{t} \in \widetilde{G}(A)$ be a lift of $t$ (as exists after possibly passing to a further local extension of $A$). Evidently $w\widetilde{t}w^{-1} \cdot \widetilde{t}^{-1} \in (\ker \pi)(A)$. Since $Z_{G_s}(t_s)$ is connected, Lemma~\ref{lemma:Weyl-elements-in-identity-component} shows that $w$ centralizes $\widetilde{t}_s$, so the special fiber of $w\widetilde{t}w^{-1} \cdot \widetilde{t}^{-1}$ is trivial. But now the hypotheses imply that the map $(\ker \pi)(A) \to (\ker \pi)(k(s))$ is injective. Indeed, if $A$ is equicharacteristic then this is true of any $A$-group scheme of multiplicative type, while if $|\pi_1(\sD(G))|$ is invertible in $A$ then $\ker \pi$ is etale. So another application of Lemma~\ref{lemma:Weyl-elements-in-identity-component} shows that $w$ has a representative lying in $Z_{G_\eta}(t_\eta)$, as desired. \smallskip

For the final claim of the lemma, we can work locally and spread out to assume that $S$ is noetherian. In this case the result follows from the first claim of this proposition and Lemma~\ref{lemma:ordinary-locus-is-locally-constructible}, using that any constructible set closed under generization is open.
\end{proof}

\subsection{(Strongly) pure sections}\label{subsection:pure}

Next we define the notion of purity for sections of $G$. The pure sections are the only sections which have a hope of having flat centralizers. In fact, we will only establish flatness results for the centralizers of strongly pure sections. 

\begin{definition}\label{definition:pure-and-strongly-pure}
A section $g \in G(S)$ is \textit{pure} if the function $s \mapsto \dim Z_{G_s}(g_s)$ is locally constant on $S$. We say that $g$ is \textit{strongly pure} if it is pure and fiberwise ordinary. Purity is defined for sections of $\fg$ in a completely similar way.
\end{definition}

\begin{lemma}\label{lemma:locus-of-d-dimensional-centralizer}
Fix an integer $d$ and consider the subsets
\begin{align*}
G_d &= \{g \in G: \dim Z_{G_{k(g)}}(g) = d\}, \\
G_{d\rm{-ord}} &= \{g \in G: \dim Z_{G_{k(g)}}(g) = d \text{ and } g \text{ is ordinary}\}.
\end{align*}
These are both constructible subsets of $G$, and $G_d$ is locally closed. If either locally on $S$ there exists a prime number $p$ which is nilpotent in $\Gamma(S, \sO_S)$ or $|\pi_1(\sD(G))|$ is invertible on $S$, then $G_{d\rm{-ord}}$ is open in $G_d$.

Similarly, define
\[
\fg_d = \{X \in \fg: \dim Z_{G_{k(X)}}(X) = d\}.
\]
This is a locally closed constructible subset of $G$.
\end{lemma}

\begin{proof}
The claims about $G_d$ and $\fg_d$ follow immediately from upper semicontinuity of fiber dimension \cite[IV\textsubscript{3}, Thm.\ 13.1.3]{EGA} applied to the universal centralizers, and the first claim about $G_{d\rm{-ord}}$ follows from this and Lemma~\ref{lemma:ordinary-locus-is-locally-constructible}. The remaining point follows from Proposition~\ref{prop:strong-purity-generizes}.
\end{proof}

For fixed integers $d$ and $e$, we let $G_{d, e}$ denote the set of $g \in G$ such that $g \in G_d$ and in the Jordan decomposition $g_{\overline{k(g)}} = tu$ we have $\dim Z_{G_{\overline{k(g)}}}(t) = e$. Moreover, let $G_{d, e\rm{-ord}} = G_{d, e} \cap G_{\rm{ord}}$. We define $\fg_{d, e}$ similarly. The following lemma will be improved slightly in Corollary~\ref{cor:openness-of-d-e-locus}.

\begin{lemma}\label{lemma:locus-of-d-e-elements}
Each set $G_{d, e}$, $G_{d, e\rm{-ord}}$, and $\fg_{d, e}$ is constructible. If either locally on $S$ there exists a prime number which is nilpotent in $\Gamma(S, \sO_S)$ or $|\pi_1(\sD(G))|$ is invertible on $S$, then $G_{d,e\rm{-ord}}$ is open in $G_{d,e}$.
\end{lemma}

\begin{proof}
We deal only with the group case, the Lie algebra case being similar. By Lemma~\ref{lemma:ordinary-locus-is-locally-constructible} and Proposition~\ref{prop:strong-purity-generizes}, it suffices to show that $G_{d, e}$ is constructible. Working etale-locally on $S$, we may and do assume that there exists a split maximal torus $T$ in $G$ contained in a Borel subgroup $B$, and by the Existence and Isomorphism Theorems \cite[Thms.\ 6.1.16 and 6.1.17]{Conrad} we may assume that $S$ is the spectrum of an excellent Dedekind domain or a field. \smallskip

For each closed subsystem $\Phi' \subset \Phi$, we define $T_{\Phi'}$ as in the proof of Lemma~\ref{lemma:ordinary-locus-is-locally-constructible}, so that $T_{\Phi'}$ is a locally closed subset of $T$; we will regard it as a subscheme with the reduced structure. By passing to a ramified extension of $S$, we may assume that whenever $(T_{\Phi'})_\eta \neq \emptyset$ there is some rational point $t_{\Phi'} \in T_{\Phi'}(K)$, where $K$ is the function field of $S$. Let $U$ be an open subscheme of $S$ such that there exists $\widetilde{t}_{\Phi'} \in T_{\Phi'}(U)$ extending $t_{\Phi'}$ whenever $(T_{\Phi'})_\eta \neq \emptyset$. By shrinking $U$, we will assume moreover that for every $s \in U$ and closed subsystem $\Phi' \subset \Phi$ we have $(T_{\Phi'})_s \neq \emptyset$ if and only if $(T_{\Phi'})_\eta \neq \emptyset$. Shrinking $U$ even further, we may assume that $Z_{G \times_S U}(\widetilde{t}_{\Phi'})$ is a smooth $U$-group scheme. To prove constructibility, we need only prove it separately over $U$ and over $S - U$; in particular, we may and do assume that for any closed subsystem $\Phi' \subset \Phi$, there exists a section $\widetilde{t}_{\Phi'} \in T_{\Phi'}(S)$ such that $Z_G(\widetilde{t})_{\Phi'}$ is a smooth $S$-group scheme. \smallskip

Let $r$ be the rank of $G$, and let $\Phi'$ be such that $T_{\Phi'} \neq \emptyset$. Set $e = |\Phi'| + r$, so that $e$ is equal to the dimension of any reductive group scheme of rank $r$ with root system $\Phi'$. If $k$ is a fixed algebraically closed field over $S$ then for every $t, t' \in T_{\Phi'}(k)$ we have $Z_{G_k}(t)^0 = Z_{G_k}(t')^0$ and $\dim Z_{G_k}(t) = e$, both of which one may see by considering the intersections of $Z_{G_k}(t)$ and $Z_{G_k}(t')$ with an open cell. Thus whenever $T_{\Phi'} \neq \emptyset$, the relative identity component $Z_{\Phi'} = Z_G(\widetilde{t}_{\Phi'})^0$ is independent of choice of section $\widetilde{t}_{\Phi'}$, and $Z_{\Phi'}$ is a reductive $S$-group scheme, and it is a closed $S$-subgroup scheme of $G$ by \cite[Thm.\ 5.3.5]{Conrad}. \smallskip

Let $B_{\Phi'} = Z_{\Phi'} \cap B$, so that $B_{\Phi'}$ is a Borel subgroup of $Z_{\Phi'}$ (as may be checked on fibers). Let $U_{\Phi'}$ be the unipotent radical of $B_{\Phi'}$. Let $g: U_{\Phi'} \to G$ be the natural inclusion, and consider the schematic centralizer $Z_{G \times_S U_{\Phi'}}(g)$, a $U_{\Phi'}$-group scheme which is a finite type closed subscheme of $G \times_S U_{\Phi'}$ by \cite[Lem.\ 2.1]{Integral-Springer}. By upper semicontinuity of fiber dimension \cite[IV\textsubscript{3}, Thm.\ 13.1.3]{EGA}, there is a locally closed subset $U_{\Phi', d} \subset U_{\Phi'}$ over which the fibers of $Z_{G \times_S U_{\Phi'}}(g) \to U_{\Phi'}$ are $d$-dimensional. Clearly we have $Z_{\Phi'} \times U_{\Phi', d} \subset G_{d, e}$. \smallskip

Now let $\mu_B: G \times B \to G$ via
\[
\mu_B(g, b) = gbg^{-1}
\]
as in the proof of Lemma~\ref{lemma:ordinary-locus-is-locally-constructible}. If $k$ is any algebraically closed field and $b \in B(k)$, say $b = tu$ for $t \in T(k)$ and $u \in U(k)$, \cite[Lem.\ 2.12]{SteinbergReg} shows that $b$ is $B(k)$-conjugate to an element of the form $b' = tu'$, where $u' \in U(k)$ commutes with $t$, and this is evidently the Jordan decomposition of $b'$. If $b'$ lies over a point of $G_{d, e}$ then by definition we have $\dim Z_G(t) = e$, so $t \in T_{\Phi'}(k)$ for some closed subsystem $\Phi' \subset \Phi$ such that $e = |\Phi'| + r$. Moreover, by definition $u'$ lies over a point of $U_{\Phi', d}$. Since each $g \in G(k)$ is conjugate to an element of $B(k)$, it follows that
\[
G_{d,e} = \mu_B\left(\bigcup_{\Phi' \subset \Phi} T_{\Phi'} \times U_{\Phi', d}\right),
\]
Thus by Chevalley's theorem, $G_{d, e}$ is constructible.
\end{proof}

\section{Relative Jordan decomposition}\label{section:relative-jordan-decomposition}

\subsection{Overview of the idea}

Over algebraically closed fields, it is very convenient to have a Jordan decomposition available in order to reduce questions about general centralizers to separate questions about semisimple centralizers and unipotent (or nilpotent) centralizers. The main result which we aim to prove in this section is the following. We remind the reader of our convention that the generic point of an irreducible scheme is always denoted by $\eta$.

\begin{theorem}\label{theorem:relative-jordan-decomposition-over-normal-base}
Let $S$ be a normal integral scheme and let $G$ be a reductive $S$-group scheme.
\begin{enumerate}
    \item Let $g \in G(S)$ be a strongly pure section of $G$ with pure semisimple part, and for every point $s \in S$ let $g_{\overline{s}} = t_s u_s$ be the Jordan decomposition of $g_{\overline{s}}$. Suppose that the Jordan decomposition of $g_{\overline{\eta}}$ is $k(\eta)$-rational. Then there exist unique sections $t, u \in G(S)$ such that $g = tu$ induces the Jordan decomposition on every $S$-fiber.
    \item Suppose that no residue characteristic of $S$ is a torsion prime for $G$. Let $X \in \fg(S)$ be a strongly pure section of $\fg$ with pure semisimple part, and for every point $s \in S$ let $X_{\overline{s}} = (X_{\rm{ss}})_s + (X_{\rm{n}})_s$ be the Jordan decomposition of $X_{\overline{s}}$. Suppose that the Jordan decomposition of $X_{\overline{\eta}}$ is $k(\eta)$-rational. Then there exist unique sections $X_{\rm{ss}}, X_{\rm{n}} \in \fg(S)$ such that $X = X_{\rm{ss}} + X_{\rm{n}}$ induces the Jordan decomposition on every $S$-fiber.
\end{enumerate}
\end{theorem}

The proof of Theorem~\ref{theorem:relative-jordan-decomposition-over-normal-base} will take up most of this section. We will begin with the case that $S$ is the spectrum of a DVR with algebraically closed residue field. In this case we will show that a ``fake Jordan decomposition" exists for any $g \in G(S)$, and that this decomposition satisfies a number of desirable properties. For details, see Theorems~\ref{theorem:relative-jordan-decomposition} and \ref{theorem:relative-additive-jordan-decomposition}; we expect that the extra generality of these results might be useful.\smallskip

If we want any analogue of the Jordan decomposition (even a ``fake" one) in a relative setting, it is important to ensure that the centralizer of the semisimple part be reductive; in particular, it is important that it be \textit{flat}. We will therefore pay special attention to purity properties of any putative decomposition. (Recall that purity is defined in Definition~\ref{definition:pure-and-strongly-pure}.) \smallskip

Throughout this section, $G$ is a reductive group scheme over a ring $A$. As in the classical case, to prove our relative Jordan decomposition we will first consider a Jordan decomposition for matrices and then pass to the general case using an embedding into $\GL_n$. Before getting to this, we warn the reader that purity is not a functorial property, as the following example shows.

\begin{example}
Let $k$ be a field and define $g \in \GL_2(k[\![x]\!])$ by
\[
g = \begin{pmatrix} \alpha & 0 \\ 0 & 1 + x \end{pmatrix},
\]
where $\alpha \in k^\times$ is not equal to $1$. We note that $g$ is (strongly) pure, and even fiberwise regular. However, if $i: \GL_2 \to \GL_3$ is the inclusion of $\GL_2$ into the ``upper left block" of $\GL_3$, then $i(g)$ is not pure; it has regular generic fiber, but the centralizer of its special fiber has dimension $5$.
\end{example}

This example shows that even if one only wants a Jordan decomposition for sections satisfying the hypotheses of Theorem~\ref{theorem:relative-jordan-decomposition-over-normal-base}, one should try to find a decomposition in $\GL_n$ for sections which are not necessarily pure.\smallskip

The argument proceeds as follows: first, construct a (non-canonical) fake Jordan decomposition inside $\GL_n$ over a henselian DVR $A$ with algebraically closed residue field. That is, for a given $g \in \GL_n(A)$, find $t, u \in \GL_n(A)$ such that $t$ is pure and $g_s = t_s u_s$ is the Jordan decomposition of $g_s$ in $\GL_n(k(s))$ (where $s$ is the closed point of $\Spec A$, as in our usual convention). In fact, our construction of a fake Jordan decomposition will satisfy a number of other desirable properties, and we will expend quite a bit of effort to establish these. Next, using some of the properties of this decomposition, we will show that if $G$ is realized as a closed $A$-subgroup scheme of $\GL_n$ in some way and the above $g$ lies in $G(A)$, then the sections $t$ and $u$ both lie in $G(A)$. From there, a number of good properties of this decomposition will be deduced, and we will bootstrap from the fake decomposition to an honest decomposition under the hypotheses of Theorem~\ref{theorem:relative-jordan-decomposition-over-normal-base}. \smallskip

The idea for the construction of the fake Jordan decomposition is very simple, and we illustrate it in the simplest case, in which $g$ is a diagonal element of $\GL_n(k[\![x]\!])$ for a field $k$. In this case, say $g = \diag(a_1, \dots, a_n)$ for some $a_i \in k[\![x]\!]^\times$. Then $g = tu$, where $t = \diag(a_1(0), \dots, a_n(0))$ and $u = \diag(a_1/a_1(0), \dots, a_n/a_n(0))$. The two fibers of $t$ are killed by precisely the same roots of the diagonal torus (so $t$ is pure), while $u$ has special fiber equal to the identity matrix. The idea for a general (not necessarily equicharacteristic) henselian DVR $A$ is very similar to this one, using an arbitrary set-theoretic section to the residue class map in place of the inclusion $k \to k[\![x]\!]$.

\subsection{Construction and proof}

As in the proof of the ordinary Jordan decomposition, we begin with a version of the Jordan decomposition for $\GL_n$.

\begin{lemma}\label{lemma:relative-jordan-decomposition-for-matrices}
Let $A$ be a henselian local ring with maximal ideal $\fm$ and algebraically closed residue field $k$. Let $V$ be a finite free $A$-module and let $g: V \to V$ be an invertible $A$-linear map. Then there exists a decomposition $g = t u$ for commuting $A$-linear maps $t, u: V \to V$ such that
\begin{enumerate}
    \item\label{item:gln-jordan-special-fiber} $g_s = t_s u_s$ is the multiplicative Jordan decomposition of $g_s$\footnote{Here $s$ is the closed point of $\Spec(A)$ as in the Notation and conventions section.};
    \item\label{item:gln-jordan-torus-character-vanishing} there is a maximal $A$-torus of $\GL(V)$ containing $t$, and $t$ is pure;
    \item\label{item:gln-jordan-polynomial} there exists a polynomial $p(X) \in A[X]$ with constant term $0$ such that $t = p(g)$; and
    \item\label{item:gln-jordan-generic-fiber} if moreover $A$ is a DVR and $g_{\overline{\eta}} = t' u'$ is the Jordan decomposition of $g_{\overline{\eta}}$, then there is some polynomial $q(X) \in k(\overline{\eta})[X]$ with constant term $0$ such that $t_{\overline{\eta}} = q(t')$.
\end{enumerate}
\end{lemma}

\begin{proof}
The proof of existence is entirely similar to the case over fields, with only a few changes; to convince the reader of this, we simply repeat the argument with the necessary changes. Since $k$ is algebraically closed, we may decompose the characteristic polynomial of $g_s$ as
\[
\det(X - g_s) = \prod_{i=1}^{m} (X - a_i)^{r_i},
\]
where $a_i \in k^\times$ are pairwise distinct units. Since $A$ is henselian, this lifts to a factorization
\[
\det(X - g) = \prod_{i=1}^{m} p_i(X),
\]
where each $p_i(X) \in A[X]$ is a monic polynomial reducing modulo $\fm_A$ to $(X - a_i)^{r_i}$. Let $V_i = \ker(p_i(g))$ for each $1 \leq i \leq m$. By the general form of Cayley-Hamilton, $V$ is a module for the ring $A[X]/(\det(X - g))$. There is a natural homomorphism 
\[
\vp: A[X]/(\det(X - g)) \to \prod_{i=1}^m A[X]/(p_i(X))
\]
between $A$-algebras which are finite free modules over $A$. Note that $\vp$ is an isomorphism; by freeness, this can be checked modulo $\fm$. If $i \neq j$, then $p_i(X)$ is a unit in $A[X]/(p_j(X))$, as may be checked modulo $\fm$ since $A[X]/(p_j(X))$ is $A$-finite. This ring-theoretic decomposition establishes the decomposition $V = \bigoplus_{i=1}^m V_i$. In particular, since $A$ is local we see that each $V_i$ is finite free. \smallskip

Now fix a (set-theoretic) section $\sigma: k \to A$ of the residue class map. By the same argument from the previous paragraph, there exists a polynomial $p(X) \in A[X]$ with constant term $0$ such that
\[
p(X) \equiv \sigma(a_i) \pmod{p_i(X)}
\]
for all $1 \leq i \leq m$. Then $t := p(g)$ acts on each $V_i$ as multiplication by $\sigma(a_i)$, and so $t$ has semisimple fibers over $\Spec A$. Moreover, it is clear from the construction that $g_s = t_s \cdot (t_s^{-1}g_s)$ is the multiplicative Jordan decomposition of $g_s$. So it remains to establish (\ref{item:gln-jordan-torus-character-vanishing}) and (\ref{item:gln-jordan-generic-fiber}) for this $t$. \smallskip

To prove (\ref{item:gln-jordan-torus-character-vanishing}), note that by choosing bases for each $V_i$, we may find a (split) maximal $A$-torus $T$ of $\GL(V)$ containing $t$. Since $t$ acts by a constant $\sigma(a_i)$ on each $V_i$ by definition, and since $\sigma(a_i)_s \neq \sigma(a_j)_s$ when $i \neq j$, the function
\[
s \mapsto \{\alpha \in \Phi(\GL(V), T): \alpha_s(t_s) = 1\}
\]
is constant on $\Spec A$ by the usual explicit description of the roots of $(\GL(V), T)$. It is therefore clear that $t$ is pure. \smallskip

Finally, for (\ref{item:gln-jordan-generic-fiber}), suppose that $A$ is a DVR and let $\overline{\eta}$ be a geometric point of $\Spec A$ lying over the generic point $\eta$. Let $g_{\overline{\eta}} = t' u'$ be the Jordan decomposition of the fiber $g_{\overline{\eta}}$. Let $L$ be a finite extension of $k(\eta)$ such that the polynomials $p_i(X)$ above are split in $L[X]$, so that in particular $t', u' \in G(L)$, and let $B$ be the localization at a maximal ideal of the normalization of $A$ in $L$. Since $A$ is a DVR, $B$ is also a DVR by the Krull-Akizuki theorem \cite[Thm.\ 11.7]{Matsumura}. By Gauss' lemma, each $p_i(X)$ splits into linear factors lying in $B[X]$. By the construction above and the usual construction of the Jordan decomposition, for each $1 \leq i \leq m$ there exists a basis $v_{i1}, \dots, v_{in_i}$ of $V_i \otimes_A L$ and units $c_{ij} \in B^\times$ such that $t' \cdot v_{ij} = c_{ij} v_{ij}$ for all $i, j$. Moreover, $c_{ij}$ is congruent to $\sigma(a_i)$ for all $i, j$, so in particular $c_{ij} \neq c_{i'j'}$ whenever $i \neq i'$. Now it suffices to show that there exists a polynomial $q(X) \in L[X]$ such that $q(c_{ij}) = \sigma(a_i)$ for all $i, j$ and $q(0) = 0$. Since $c_{ij} \neq c_{i'j'}$ whenever $i \neq i'$ and $L$ is infinite, such $q$ indeed exists.
\end{proof}

Passage to the general case is somewhat involved. We will need the following lemma to deduce purity in Theorem~\ref{theorem:relative-jordan-decomposition} from the purity statement of Lemma~\ref{lemma:relative-jordan-decomposition-for-matrices}.

\begin{lemma}\label{lemma:purity-checked-after-extension}
Let $S$ be a scheme and let $i: G \to G'$ be a closed embedding of reductive $S$-group schemes. Let $t \in G(S)$ be a fiberwise semisimple section which lies in some maximal $S$-torus $T$ of $G$. If $i(t)$ is pure in $G'(S)$, then $t$ is pure in $G(S)$.
\end{lemma}

\begin{proof}
Purity is a fibral condition, so by working locally in the etale topology we may assume that $S = \Spec A$ for a strictly henselian local ring $A$. By \cite[Lem.\ 2.2.4]{Conrad}, the centralizer $Z_{G'}(i(T))$ is a smooth affine $A$-group scheme, so because it has connected reductive fibers (see \cite[Exp.\ XII, Th\'eor\`eme 6.6]{SGA3II} for the connectedness claim) it is reductive; it is necessarily split since $A$ is strictly henselian. Thus there is a split maximal $A$-torus $T'$ of $G'$ containing $i(T)$. We claim first that the image of the restriction map $\Phi' := \Phi(G', T') \to X(T)$ contains $\Phi := \Phi(G, T)$. For this, we may pass to the special fiber of $A$ and thus assume that $A$ is a field. Consider the decomposition
\[
\fg' = \ft' \oplus \bigoplus_{\alpha' \in \Phi'} \fg'_{\alpha'}
\]
of $\fg' = \Lie G'$ into weight spaces for the adjoint action of $T'$, where $\ft' = \Lie T'$. The decomposition of $\fg'$ into weight spaces for $T$ is then obtained by restricting each $\alpha' \in \Phi'$ to $T$. Since $\fg$ is a $T$-subrepresentation of $\fg'$ we obtain the claim easily. \smallskip

For every $s \in \Spec A$, $Z_{G_s}(t_s)$ contains as an open subvariety the direct product under multiplication of $T_s$ and those root groups $U_{\alpha_s}$ such that $\alpha_s(t_s) = 1$; this follows from Lie algebra considerations. Suppose $\alpha \in \Phi$ is a root and let $\alpha' \in \Phi'$ be such that $\alpha'|_T = \alpha$ (as exists by the previous paragraph). Since $\alpha'(t) = \alpha(t)$, purity of $i(t)$ implies that the function
\[
s \mapsto \{\alpha \in \Phi: \alpha_s(t_s) = 1\}
\]
is constant on $\Spec A$; i.e., $t$ is pure in $G(A)$.
\end{proof}

The following flatness criterion is a variant of \cite[Prop.\ 6.1]{Gan-Yu} originally appearing in \cite[Lem.\ 4.4]{Booher}. For the convenience of the reader, we will prove it. For a slightly different proof, see \cite[Lem.\ 4.4]{Booher}.

\begin{lemma}\label{lemma:boohers-lemma}
Let $S$ be a Dedekind scheme and let $Y$ be a finite type $S$-scheme. Suppose
\begin{enumerate}
    \item\label{item:equidim-locally-constant} Every fiber $Y_s$ is equidimensional and the function $s \mapsto \dim Y_s$ is locally constant on $S$.
    \item\label{item:reduced} Every fiber $Y_s$ is reduced.
    \item\label{item:enough-sections} For every closed point $s \in S$ and every irreducible component $Z$ of $Y_s$, there is a section $y \in Y(\sO_{S,s})$ such that $Z$ is the only irreducible component of $Y_s$ containing $y_s$. (For example, this holds if every fiber $Y_s$ is irreducible and $Y(S) \neq \emptyset$.)
\end{enumerate}
Then $Y$ is $S$-flat.
\end{lemma}

\begin{proof}
Working Zariski-locally, we can assume $S$ is the spectrum of a DVR $A$ with closed point $s$ and generic point $\eta$. Let $Z$ be the schematic closure of the generic fiber $Y_\eta$ in $Y$. Since $A$ is a DVR, $Z$ is $S$-flat, and thus it suffices to show that the morphism $Z \to Y$ is an isomorphism. Since $Z$ is $S$-flat, the fibral isomorphism criterion \cite[IV\textsubscript{4}, Cor.\ 17.9.5]{EGA} allows us to reduce to proving that the morphism $Z_s \to Y_s$ of special fibers over $S$ is an isomorphism. \smallskip

Since $Z$ is flat and $Y_\eta$ is equidimensional of dimension $d$, the same is true of $Z_s$. As the underlying topological space of $Z$ is equal to the topological closure of $Y_\eta$ in $Y$, (\ref{item:enough-sections}) implies that for every irreducible component $Y_0$ of $Y_s$, there is some point $z \in Z_s$ such that $Y_0$ is the only irreducible component of $Y_s$ containing $z$. The local ring $\sO_{Y_s, z}$ is an integral domain of dimension $d$ by (\ref{item:equidim-locally-constant}) and the choice of $z$, so the map to its $d$-dimensional quotient $\sO_{Z_s, z}$ is necessarily an isomorphism. Thus we see that $Z_s$ contains every generic point of $Y_s$, so that $Z_s \to Y_s$ is a homeomorphism. By (\ref{item:reduced}), $Y_s$ is reduced, so the closed embedding $Z_s \to Y_s$ is indeed an isomorphism.
\end{proof}

\begin{lemma}\label{lemma:semisimple-centralizer-is-reductive}
Let $S$ be a reduced scheme and let $G \to S$ be a reductive group scheme.
\begin{enumerate}
    \item If $g \in G(S)$ is strongly pure and fiberwise semisimple, then $Z_G(g)$ is a reductive group scheme.
    \item Suppose that no residue characteristic of $S$ is a torsion prime for $G$. If $X \in \mathfrak{g}(S)$ is pure and fiberwise semisimple, then $Z_G(X)$ is a reductive group scheme.
\end{enumerate}
\end{lemma}

\begin{proof}
By localizing and spreading out (using Lemma~\ref{lemma:locus-of-d-e-elements} and standard constructibility arguments), we may and do assume that $S$ is noetherian. By the valuative criterion of flatness \cite[IV\textsubscript{3}, Thm.\ 11.8.1]{EGA}, we may and do further assume that $S$ is the spectrum of a DVR. For the first point, note that for every $s \in S$, the fiber $Z_{G_s}(g_s)$ is connected and reductive by \cite[Thm.\ 2.15]{SteinbergTorsion}. The same is therefore true of $Z_{G_s}(g_s)/Z(G_s)$, so $Z_G(g)/Z(G)$ is $S$-flat by Lemma~\ref{lemma:boohers-lemma}, and it follows that $Z_G(g)$ is $S$-flat since the morphism $Z_G(g) \to Z_G(g)/Z(G)$ is a $Z(G)$-torsor. Thus $Z_G(g)$ is smooth, as can be checked fibrally \cite[IV\textsubscript{3}, Prop.\ 17.8.2]{EGA}, and it has connected reductive fibers, whence it is a reductive group scheme by definition. For the second point, the argument is similar: for every $s \in S$, the fiber $Z_{G_s}(X_s)$ is connected and reductive by \cite[Prop.\ 3.7]{Integral-Springer}, so we can apply Lemma~\ref{lemma:boohers-lemma} in a completely similar way to conclude that $Z_G(X)$ is a reductive group scheme.
\end{proof}

Finally, we are ready to state and prove the existence of a general ``fake Jordan decomposition''. We emphasize that the sections $t$ and $u$ in the following theorem are highly non-canonical.

\begin{theorem}\label{theorem:relative-jordan-decomposition}
Let $A$ be a henselian DVR with algebraically closed residue field $k$, and let $G$ be a reductive group scheme over $A$. If $g \in G(A)$ is any section, then there exist $t, u \in G(A)$ such that
\begin{enumerate}
    \item\label{item:general-jordan-special-fiber} $g_s = t_s u_s$ is the Jordan decomposition of $g_s$, where $s$ is the closed point of $\Spec A$,
    \item\label{item:general-jordan-pure} $t$ is pure and fiberwise semisimple,
    \item\label{item:general-jordan-commute} $Z_G(g) \subset Z_G(t)$, and if $g_{\overline{\eta}} = t' u'$ is the Jordan decomposition of $g_{\overline{\eta}}$ then $Z_{G_{\overline{\eta}}}(t') \subset Z_{G_{\overline{\eta}}}(t_{\overline{\eta}})$.
\end{enumerate}
\end{theorem}

\begin{proof}
Since $A$ is a DVR (and in particular Dedekind), there is an embedding $i: G \to \GL_n$ for some $n$; we will use this embedding to regard $G$ as a closed subgroup scheme of $\GL_n$. By Lemma~\ref{lemma:relative-jordan-decomposition-for-matrices}, there exist $t, u \in \GL_n(A)$ such that there is a decomposition $g = tu$ satisfying the conclusions of that lemma. If $\overline{\eta}$ is a geometric generic point of $\Spec A$, then $G(A) = \GL(V)(A) \cap G(k(\overline{\eta}))$, so to check that $t$ and $u$ lie in $G(A)$, it suffices to show that they lie in $G(k(\overline{\eta}))$. For this, we may use \cite[I, 3.8, Cor.]{Borel} as in the proof of \cite[I, 4.4, Thm.(1)]{Borel}: by the reference, we have
\[
G(k(\ov\eta)) = \{g \in \GL(V)(k(\ov\eta))\colon \rho_gJ = J\},
\]
where $J$ is the ideal in the coordinate ring $R$ of $\GL(V)$ defining $G$ and $\rho$ refers to the induced representation of $G$ on $R$. Since the polynomial $p(X)$ in Lemma~\ref{lemma:relative-jordan-decomposition-for-matrices}(\ref{item:gln-jordan-polynomial}) has constant term $0$, for any $g \in G(k(\ov\eta))$ we have
\[
\rho_tJ = p(\rho_g)J = J,
\]
so $t \in i(G)(k(\overline{\eta}))$, and thus also $u \in i(G)(k(\overline{\eta}))$. So (\ref{item:general-jordan-special-fiber}) follows from Lemma~\ref{lemma:relative-jordan-decomposition-for-matrices}(\ref{item:gln-jordan-special-fiber}). \smallskip

For (\ref{item:general-jordan-pure}), fiberwise semisimplicity follows from Lemma~\ref{lemma:relative-jordan-decomposition-for-matrices}(\ref{item:gln-jordan-torus-character-vanishing}). To prove the purity claim, we may pass to any DVR which is a local extension of $A$. Thus by Lemma~\ref{lemma:dvr-covering-by-borels} we may assume that there is a Borel $A$-subgroup $B$ of $G$ containing $t$. Let $T$ be a (split) maximal $A$-torus of $B$ and let $t_0$ be the component of $t$ lying in $T$ under the decomposition $B = T \ltimes U$. Using \cite[Lem.\ 2.12]{SteinbergReg}, we find that each fiber of $t$ is conjugate to the corresponding fiber of $t_0$, so the fibers of $Z_G(t_0)$ and $Z_G(t)$ have the same dimension, and the same is true of the fibers of $Z_{\GL_n}(i(t_0))$ and $Z_{\GL_n}(i(t))$. Thus $i(t_0)$ is pure in $\GL_n(A)$, and Lemma~\ref{lemma:purity-checked-after-extension} shows that $t_0$ is pure in $G(A)$. Thus $t$ is also pure in $G(A)$, as desired. \smallskip

Finally, the proof of (\ref{item:general-jordan-commute}) follows from Lemma~\ref{lemma:relative-jordan-decomposition-for-matrices}(\ref{item:gln-jordan-polynomial}), using the same argument as in the proof of \cite[Lem.\ 2.6]{Integral-Springer}; we leave the details to the reader.
\end{proof}

\begin{cor}\label{corollary:relative-jordan-pure}
Let $A$ be a henselian DVR with algebraically closed residue field $k$, and let $G$ be a reductive group scheme over $A$. Let $g \in G(A)$ be a section and for each geometric point $\overline{x}$ of $\Spec A$ let $g_{\overline{x}} = t_{\overline{x}} u_{\overline{x}}$ be the Jordan decomposition of $g_{\overline{x}}$. Suppose that
\begin{itemize}
    \item $Z_{G_{\overline{s}}}(t_{\overline{s}})$ is connected,
    \item either $A$ is equicharacteristic or $|\pi_1(\sD(G))|$ is invertible in $A$,
    \item $\dim Z_{G_{\overline{s}}}(t_{\overline{s}}) = \dim Z_{G_{\overline{\eta}}}(t_{\overline{\eta}})$.
\end{itemize}
Then there exists a finite local extension of DVRs $A \to A'$ and $t', u' \in G(A')$ such that $t'_{\overline{x}} = t_{\overline{x}}$ and $u'_{\overline{x}} = u_{\overline{x}}$ for all geometric points $\overline{x}$ of $\Spec A'$.
\end{cor}

\begin{proof}
Let $t_0, u_0 \in G(A)$ be sections such that $g = t_0 u_0$ is a fake Jordan decomposition as in Theorem~\ref{theorem:relative-jordan-decomposition}. By Theorem~\ref{theorem:relative-jordan-decomposition}(\ref{item:general-jordan-pure}), we have the equality $\dim Z_{G_{\overline{\eta}}}(t_{0,\overline{\eta}}) = \dim Z_{G_{\overline{\eta}}}(t_{\overline{\eta}})$. Moreover, by Theorem~\ref{theorem:relative-jordan-decomposition}(\ref{item:general-jordan-commute}) we have $Z_{G_{\overline{\eta}}}(t_{0, \overline{\eta}}) \subset Z_{G_{\overline{\eta}}}(t_{\overline{\eta}})$. By Proposition~\ref{prop:strong-purity-generizes}, $t_0$ is strongly pure, so we see that this inclusion is an equality. In particular, $t_{0, \overline{\eta}} \in Z(Z_{G_{\overline{\eta}}}(t_{\overline{\eta}}))$. \smallskip

Since $t_0$ is strongly pure, Lemma~\ref{lemma:semisimple-centralizer-is-reductive} shows that $Z_G(t_0)$ is a reductive group scheme over $A$. Replacing $G$ by $Z_G(t_0)$, the first paragraph shows that we may assume that the semisimple part of each geometric fiber of $g$ is central in $G$. Further replacing $A$ by a localization of its normalization in some finite extension $K'$ of $K$, we may assume that the Jordan decomposition $g_\eta = t_\eta u_\eta$ of $g_\eta$ is $k(\eta)$-rational, and by Lemma~\ref{lemma:dvr-covering-by-borels} we may assume that there is a Borel $A$-subgroup $B$ of $G$ such that $g \in B(A)$. Write $B = T \ltimes U$ for a maximal torus $T$ of $B$ and decompose $g = t' u'$ (uniquely) for $t' \in T(A)$ and $u' \in U(A)$. \smallskip

Now fix a point $x \in \Spec A$ and let $g_x = t_x u_x$ be the Jordan decomposition of $g_x$ (which is rational by hypothesis). Note that functoriality of the Jordan decomposition implies $t_x, u_x \in B_x$. Since $t_x$ is \textit{central} in $G_x$, it lies in $T_x$. Since every unipotent element of $B(k(x))$ lies in $U(k(x))$, we see that $t'_x = t_x$ and $u'_x = u_x$, completing the proof.
\end{proof}

The following version of the additive Jordan decomposition is proved very similarly; we omit the proof for reasons of length and because it has no new features. The restriction on the residue characteristic comes from the restriction in Lemma~\ref{lemma:semisimple-centralizer-is-reductive}, but we do not know whether it is necessary.

\begin{theorem}\label{theorem:relative-additive-jordan-decomposition}
Let $A$ be a henselian DVR with algebraically closed residue field $k$, and let $G$ be a reductive group scheme over $A$ with Lie algebra $\fg$ such that $\chara k$ is not a torsion prime for $G$. If $X \in \fg$ then there exist $X_{\rm{ss}}, X_{\rm{n}} \in \fg$ such that
\begin{enumerate}
    \item $X_s = (X_{\rm{ss}})_s + (X_{\rm{n}})_s$ is the additive Jordan decomposition of $X_s$,
    \item $X_{\rm{ss}}$ is pure and fiberwise semisimple, and
    \item $Z_G(X) \subset Z_G(X_{\rm{ss}})$, and if $X_{\overline{\eta}} = X'_{\rm{ss}} + X'_{\rm{n}}$ is the Jordan decomposition of $X_{\overline{\eta}}$ then we have $Z_{G_{\overline{\eta}}}(X'_{\rm{ss}}) \subset Z_{G_{\overline{\eta}}}((X_{\rm{ss}})_{\overline{\eta}})$.
\end{enumerate}
Moreover, suppose that $(X_{\rm{ss}})_s$ has connected centralizer. Let $X_{\overline{\eta}} = X'_{\rm{ss}} + X'_{\rm{n}}$ be the Jordan decomposition of $X_{\overline{\eta}}$ and suppose that $\dim Z_{G_s}((X_{\rm{ss}})_s) = \dim Z_{G_{\overline{\eta}}}(X'_{\rm{ss}})$. Then there exists a finite extensions of DVRs $A \to B$ and $X_{\rm{ss}, 1}, X_{\rm{n}, 1} \in \fg \otimes_A B$ such that $(X_{\rm{ss}, 1})_s = (X_{\rm{ss}})_s$ and $(X_{\rm{ss}, 1})_{\overline{\eta}} = X'_{\rm{ss}}$.
\end{theorem}

\begin{proof}[Proof of Theorem~\ref{theorem:relative-jordan-decomposition-over-normal-base}]
We will only establish the first point, the proof of the second point being entirely similar. Because of the uniqueness claim, we may work locally to assume that $S = \Spec A$ is affine. By spreading out (using Lemma~\ref{lemma:locus-of-d-e-elements}), we may further assume that $S$ is noetherian. By hypothesis, there is a Jordan decomposition $g_{\eta} = tu$ for $t, u \in G(k(\eta))$; we need to show that in fact $t$ and $u$ lie in $G(S)$ and that $g = tu$ induces the Jordan decomposition on every fiber. By Hartogs' lemma \cite[Thm.\ 11.5]{Matsumura}, to show the first point it suffices to show that $t$ and $u$ lie in $G(A_{\fp})$ for every height $1$ prime of $A$, so we may assume that $A$ is a DVR. Let $K = k(\eta)$ be the fraction field of $A$ and let $\pi$ be a uniformizer of $A$. By \cite[0\textsubscript{III}, Prop.\ 10.3.1]{EGA}, there exists a DVR $A'$ containing $A$ such that the residue field of $A'$ is algebraically closed and $\pi$ is a uniformizer of $A'$. We have then $A = A' \cap K$, so we may pass from $A$ to $A'$ and thus assume that $A$ has algebraically closed residue field. Moreover, if $A^h$ is the henselization of $A$ then we have $A = A^h \cap K$, so again we may pass from $A$ to $A^h$ to assume that $A$ is a henselian DVR with algebraically closed residue field. We may also pass to this case to show that $g = tu$ induces the Jordan decomposition on every fiber. But now the result follows from Corollary~\ref{corollary:relative-jordan-pure}.
\end{proof}

\subsection{Some subschemes of $G$}

We conclude this section with a few applications of the Jordan decomposition to understanding certain natural loci in $G$. Recall the subsets $G_{d,e}$ and $G_{d,e\rm{-ord}}$ of $G$ defined in the paragraph preceding Lemma~\ref{lemma:locus-of-d-e-elements}. The main point of the following two corollaries is to put natural scheme structures on $G_{d, e}$ and $G_{d,e\rm{-ord}}$. After doing so, we will see in Theorem~\ref{theorem:flat-pure-centralizer} that the universal centralizer (i.e., the centralizer of the universal point $\id_G \in G(G)$) is flat over $G_{d,e\rm{-ord}}$, and so we will see that the $G_{d,e\rm{-ord}}$ form a natural flattening stratification of $G_{\rm{ord}}$. 

\begin{cor}\label{cor:openness-of-d-e-locus}
Let $G$ be a reductive group scheme over a scheme $S$. For every pair of integers $d$ and $e$, the subset $G_{d,e}$ is locally closed in $G$. If either locally on $S$ there exists a prime number $p$ which is nilpotent in $\Gamma(S, \sO_S)$ or $|\pi_1(\sD(G))|$ is invertible on $S$, then $G_{d,e\rm{-ord}}$ is open in $G_{d, e}$.
\end{cor}

\begin{proof}
We have already seen in Lemma~\ref{lemma:locus-of-d-e-elements} that all of these sets are constructible in $G$. Let now $G_{d, \leq e}$ be the union of the sets $G_{d, e'}$ for all $e' \leq e$, so that $G_{d, \leq e}$ is constructible and $G_{d, e} = G_{d, \leq e} \setminus G_{d, \leq e-1}$. Thus by Lemma~\ref{lemma:ordinary-locus-is-locally-constructible} it suffices to show that each $G_{d, \leq e}$ is open in $G_d$, and by constructibility it suffices to show that it is closed under generization. This reduces us to proving that if $A$ is a complete DVR with algebraically closed residue field and $g \in G(A)$ is such that $\dim Z_{G_s}(g_s) = \dim Z_{G_\eta}(g_\eta)$ and we have Jordan decompositions $g_s = t_0 u_0$ and $g_{\overline{\eta}} = t_1 u_1$ then $\dim Z_{G_{\overline{\eta}}}(t_1) \leq \dim Z_{G_s}(t_s)$. \smallskip

By Corollary~\ref{corollary:relative-jordan-pure}, there exist $t, u \in G(A)$ such that $g_s = t_s u_s$ is the Jordan decomposition of $g_s$, $t$ is pure and fiberwise semisimple, and $Z_{G_{\overline{\eta}}}(t_1) \subset Z_{G_{\overline{\eta}}}(t_{\overline{\eta}})$. This gives the desired dimension inequality.
\end{proof}

\begin{prop}\label{prop:scheme-of-d-regular-elements}
Let $S$ be a scheme and let $G \to S$ be a reductive group scheme such that $\sD(G)$ is simply connected. Choose an integer $d$ and consider the functor $F_d = F_{d, G}$ on $S$-schemes given by
\[
S' \leadsto \{g \in G(S'): g \text{ is fiberwise semisimple and } Z_{G \times_S S'}(g) \text{ is $S'$-flat with fiber dimension $d$}\}.
\]
Then $F_d$ is representable by a finitely presented $S$-flat locally closed subscheme of $G$. If $Z(G)$ is smooth, then $F_d$ is smooth.
\end{prop}

\begin{proof}
It is easy to check that $F_d$ is an fpqc sheaf. By effectivity of fpqc descent for quasi-affine morphisms \cite[Exp.\ VIII, Cor.\ 7.9]{SGA1}, to prove representability of $F_d$ we may pass to any fpqc cover of $S$. Moreover, flatness, finite presentation, and smoothness all satisfy fpqc descent by \cite[IV\textsubscript{2}, Props.\ 2.5.1 and 2.7.1(vi); IV\textsubscript{4}, Cor.\ 17.7.3(ii)]{EGA}. Thus we may pass to an etale cover of $S$ to assume that $G$ is $S$-split. In this case, there is a reductive group scheme $\bG$ over $\bZ$ and an isomorphism $G \cong \bG \times_{\Spec \bZ} S$, and $F_d$ is the base change to $S$ of the analogous functor for $\bG$. So we may pass from $G$ to $\bG$ and thus assume that $S$ is reduced and noetherian. Embed the center $Z(G)$ into a torus $T$ and consider the reductive group scheme $G' = G \times^{Z(G)} T$, which has smooth center and derived group $\sD(G)$. Note that if $S'$ is an $S$-scheme and $g \in G(S')$ then $Z_{G'_{S'}}(g) = Z_{G_{S'}}(g) \times^{Z(G_{S'})} T_{S'}$ since $G \subset G'$ and $G' = G \cdot T$ on fibers; the contracted product $Z_{G_{S'}}(g) \times^{Z(G_{S'})} T_{S'}$ admits the $S'$-group scheme $Z_{G_{S'}}(g) \times T_{S'}$ as a $Z(G_{S'})$-torsor over it, so $Z_{G'_{S'}}(g)$ is flat if and only if $Z_{G_{S'}}(g)$ is flat. Thus one checks that
\[
F_{d + \dim T - \dim Z(G), G'} = F_{d, G} \times^{Z(G)} T,
\]
so that $F_{d + \dim T - \dim Z(G), G'}/T \cong F_{d, G}/Z(G)$. If $F_{d + \dim T - \dim Z(G), G'}$ is representable and $S$-smooth, then $F_{d, G}/Z(G)$ is therefore representable (and $S$-smooth) by \cite[Exp.\ VIII, Thm.\ 5.1; Exp.\ IX, Prop.\ 2.3]{SGA3II}. Moreover, the morphism $F_{d, G} \to F_{d, G}/Z(G)$ is a $Z(G)$-torsor, so because $Z(G)$ is $S$-affine it follows from effectivity of fpqc descent for quasi-affine morphisms that $F_{d, G}$ is representable, and it is $S$-flat since $Z(G)$ is $S$-flat. Thus it suffices to prove that each $F_{d, G'}$ is representable and $S$-smooth, so we may assume that $Z(G)$ is smooth. \smallskip

Consider the subset $X := G_{d, \rm{ss}}$ of $G$ defined in Section~\ref{section:constructibility-nonsense}, so that $X$ is locally closed in $G$ by Corollary~\ref{cor:openness-of-d-e-locus}. We equip $X$ with the reduced subscheme structure, and we claim that $X$ represents $F_d$; note that since $S$ was assumed noetherian, $X$ is a finitely presented locally closed subscheme of $G$. \smallskip

Since $X$ is reduced, if $i: X \to G$ is the natural inclusion then $Z_{G \times_S X}(i)$ is $X$-flat by Lemma~\ref{lemma:semisimple-centralizer-is-reductive}. In particular, the same is true for any scheme mapping to $X$, so that $h_X \subset F_d$ as functors. To show the reverse inclusion, we will first establish that $F_d$ is formally smooth. Let $A$ be a ring and let $I \subset A$ be an ideal such that $I^2 = 0$. If $g \in F_d(A/I)$, then by definition $Z_{G \times_S \Spec A/I}(g)$ is flat, and thus it is smooth since the centralizer of a semisimple element is smooth. Let $Z$ be the center of $Z_{G \times_S \Spec A/I}(g)$, so that $Z$ is a smooth multiplicative type subgroup containing $g$ by \cite[Cor.\ 3.2]{Integral-Springer}. By smoothness of the scheme of multiplicative type subgroups of $G$ \cite[Exp.\ XI, Thm.\ 4.1]{SGA3II}, there is a smooth multiplicative type subgroup $\widetilde{Z}$ of $G \times_S \Spec A$ lifting $Z$. By \cite[Exp.\ XI, Cor.\ 5.3]{SGA3II}, $Z_{G \times_S \Spec A}(\widetilde{Z})$ is a smooth subgroup scheme of $G \times_S \Spec A$. Since $\widetilde{Z}$ is smooth, there is some $\widetilde{g} \in \widetilde{Z}(A)$ lifting $g$, and we claim that the natural inclusion $Z_{G \times_S \Spec A}(\widetilde{Z}) \to Z_{G \times_S \Spec A}(\widetilde{g})$ is an isomorphism. Indeed, the source is flat and this becomes an an isomorphism upon base change to $\Spec A/I$, so the fibral isomorphism criterion shows that it is an isomorphism. Thus $\widetilde{g} \in F_d(A)$ lifts $g$ and indeed $F_d$ is formally smooth. \smallskip

Now suppose $g \in F_d(A)$ for some ring $A$. By definition, $g$ factors through $X$ set-theoretically. To check that this is the case schematically, we can work locally and spread out to assume that $A$ is noetherian local. In that case, we claim that it suffices to prove that the map $\Spec A/\mathfrak{m}_A^N \to G$ factors through $X$ for all $N$. In fact, we make the more general claim that if $\cY$ is a noetherian scheme and $\cX \subset \cY$ is a locally closed subscheme, then a map $f\colon \Spec A \to \cY$ which factors through $\cX$ set-theoretically factors through it scheme-theoretically if and only if the induced map $\Spec A/\fm_A^n \to \cY$ factors through $\cX$ for all $n$. For this, we may localize to assume $\cY = \Spec R$ and $\cX = \Spec R/I$ for some ideal $I \subset R$. Thus the map $f$ corresponds to a homomorphism $f^\sharp\colon R \to A$, and we assume that the induced map $f^\sharp_n\colon R \to A/\fm_A^n$ satisfies $I \subset \ker f^\sharp_n$ for all $n$. This means that $f^\sharp(I) \subset \fm_A^n$ for all $n$, so $I \subset \ker f^\sharp$ by the Krull intersection theorem. Thus we may pass from $A$ to $A/\fm_A^n$ for all $n$ to assume that $A$ is Artin local.

In this case, the Cohen structure theorem shows that there is some surjection $R \to A$ from a regular local ring $R$ \cite[Theorem 29.4(ii)]{Matsumura}. Using formal smoothness of $F_d$, we may successively lift $g$ to $g_N \in F_d(R/\mathfrak{m}_R^N)$ for all large $N$. These successive lifts give rise to a point $g_\infty \in G(R)$ (since $G$ is representable), and the local criterion of flatness \cite[Thm.\ 22.1]{Matsumura} shows that $Z_{G \times_S \Spec R}(g_\infty)$ is flat, i.e., that $g_\infty \in F_d(R)$. But now $g_\infty$ still factors through $X$ set-theoretically, and because $R$ is reduced it follows that it factors through $X$ \textit{schematically}. Since $\Spec A$ is a closed subscheme of $\Spec R$, we see that $g$ factors through $X$ schematically, and thus indeed $X$ represents $F_d$. Since we have already shown that $F_d$ is formally smooth and $X$ is finitely presented, it follows that in this case $X$ is smooth, proving the result.
\end{proof}

\section{Centralizers over a general base}\label{section:general-centralizers}

\subsection{Pure centralizers in good characteristic}

In this section we will prove Theorem~\ref{theorem:intro-flat-pure-centralizer} and establish a number of corollaries. In outline, the proof works as follows.
\begin{enumerate}
    \item By spreading out and using the valuative criterion of flatness \cite[IV\textsubscript{3}, Thm.\ 11.8.1]{EGA}, reduce to the case that the base scheme $S$ is the spectrum of a DVR.
    \item Using a Springer isomorphism \cite[Thm.\ 5.1]{Integral-Springer} and a relative version of the Jordan decomposition (Theorem~\ref{theorem:relative-jordan-decomposition-over-normal-base}), reduce to the case of a pure fiberwise nilpotent section $X$ of $\fg$.
    \item If $X \in \fg$ is pure and fiberwise nilpotent, then show that there is a cocharacter $\tau: \bG_m \to G$ each of whose pullbacks to the fibers over $S$ is almost associated to the corresponding fiber of $X$. In particular, each fiber of $P = P_G(\tau)$ is the instability parabolic of the corresponding fiber of $X$.
    \item For $X$ and $\tau$ as above, use the results of Section~\ref{subsection:associated-cocharacters} to show that the orbit map $\Phi: P \to \bigoplus_{n \geq 2} \fg(\tau, n)$ through $X$ is flat.
\end{enumerate}

To argue the third point in this outline, we will proceed in two steps: first, show that if $P$ is a parabolic subgroup scheme of $G$ which is the instability parabolic of some section on the generic fiber, then there exists \textit{some} pure fiberwise nilpotent $Y$ of $\fg$ (possibly after extending $A$) such that each fiber of $P$ is the instability parabolic of the corresponding fiber of $Y$ and such that there is a cocharacter $\phi: \bG_m \to P$ which is fiberwise associated to $Y$. Next, if $X$ is any pure fiberwise nilpotent section of $\fg$ and $P$ is the parabolic whose generic fiber is the instability parabolic of $X_\eta$, then we show that $X_s$ is conjugate to the special fiber $Y_s$ of the section $Y$ built above. Using this, it will follow that $P_s$ is the instability parabolic of $X_s$. \smallskip

These two steps correspond to the following two lemmas, Lemma~\ref{lemma:mcninch-existence-of-sections} and Lemma~\ref{lemma:mcninch-relative-instability-parabolic}. The second lemma is a slightly weaker form of \cite[5.6.2]{McNinch-relative-centralizer}. There is an error in the proof of \cite[5.6.2]{McNinch-relative-centralizer} owing to the fact that it relies on the false \cite[2.3.2]{McNinch-relative-centralizer}. However, the overall shape of our argument comes from the proof of this. The first lemma is a weak form of \cite[5.4, Thm.]{McNinch-relative-centralizer}, which we now prove for the convenience of the reader.

\begin{lemma}\label{lemma:mcninch-existence-of-sections}
Let $A$ be a local noetherian normal domain and let $G$ be a reductive $A$-group scheme such that $\chara k(s)$ is good for $G_s$. Let $X_1 \in \fg_\eta$ be such that there exists a parabolic $A$-subgroup scheme $P$ of $G$ whose generic fiber is the instability parabolic of $X_1$. Then there exists a generically finite local extension of local noetherian normal domains $A \to A'$ and a section $Y \in \fg(A')$ such that $Y_{\eta'}$ is geometrically conjugate to $X_1$ and $P_t$ is the instability parabolic of $Y_t$ for all $t \in \Spec A'$.\smallskip 

Moreover, if $\phi: \bG_m \to P$ is any cocharacter such that $\phi_\eta$ is almost associated to $X_1$, then we may find $Y$ such that $\phi_t$ is associated to $Y_t$ for all $t \in \Spec A'$.
\end{lemma}

\begin{proof}
First note that there is a cocharacter $\phi: \bG_m \to P$ such that $\phi_\eta$ is almost associated to $X_1$. To see this, let $T_0$ be a maximal $A$-torus of $P$ and note that Proposition~\ref{proposition:mcninch-almost-associated-cocharacters}(\ref{item:almost-associated-cocharacters-exist}) shows that there is a (unique) cocharacter $\phi_\eta: \bG_m \to T_{0,\eta}$ which is almost associated to $X_1$. Since the natural map $\pi_1(\Spec k(\eta), \overline{\eta}) \to \pi_1(\Spec A, \overline{\eta})$ is surjective by \cite[Exp.\ V, Prop.\ 8.2]{SGA1}, we see that $\Hom_A(\bG_m, T_0) = \Hom_{k(\eta)}(\bG_m, T_{0, \eta})$. Thus we have found the desired $\phi$. Note that $P = P_G(\phi)$: this can be checked on the generic fiber, where it holds by Proposition~\ref{proposition:mcninch-almost-associated-cocharacters}(\ref{item:almost-associated-cocharacters-dynamic-method}). \smallskip

Since $\phi_\eta$ is almost associated to $X_1$, there is some $g_1 \in P(k(\overline{\eta}))$ such that $\phi_\eta^{g_1}$ is associated to $X_1$. Extending $A$, we may assume that $g_1$ is a rational point. Passing from $X_1$ to $\Ad(g_1^{-1})X_1$, we may therefore assume that $\phi_\eta$ is associated to $X_1$. By definition, this means that $X_1 \in \fg_\eta(\phi_\eta, 2)$ and there exists a parabolic subgroup $Q_1$ of $G_\eta$ and a Levi subgroup $L_1$ of $Q_1$ such that $\phi_\eta$ factors through $\sD(L_1)$ and $X_1 \in \Lie L_1$ is distingushed. Let $Q$ be the parabolic $A$-subgroup of $G$ with generic fiber $Q_1$, and let $T$ be a (split) maximal $A$-torus of $Q$ through which $\phi$ factors. By \cite[Prop.\ 5.4.5]{Conrad}, there is a unique Levi $A$-subgroup $L$ of $Q$ containing $T$. \smallskip

We claim that $\phi$ factors through $\sD(L)$. To prove this claim, first note that it suffices to check this over $\eta$. In any case, the image of $\phi_\eta$ lies in $L_\eta$, and it also lies in $\sD(L_1)$. By the latter fact, the image of $\phi_\eta$ in $Q_1/\sR_u(Q_1)$ is contained in $\sD(Q_1/\sR_u(Q_1))$. Since the map $L_\eta \to Q_1/\sR_u(Q_1)$ is an isomorphism, we see that indeed $\phi_\eta$ factors through $\sD(L_\eta)$ and so $\phi$ factors through $\sD(L)$. \smallskip

We claim that for all $t \in \Spec A$, $P_L(\phi)_t$ is a distinguished parabolic subgroup of $L_t$. By Proposition~\ref{proposition:properties-of-instability-parabolic}(\ref{item:instability-parabolic-bala-carter}), $(L_1, P_{L_1}(\phi_\eta))$ is a Bala--Carter datum of $X_1$, so in particular $P_{L_1}(\phi_\eta)$ is a distinguished parabolic subgroup of $L_1$. Since $T_1$ and $T_\eta$ are both maximal tori in $Z_{G_\eta}(\phi_\eta)$, we see that they are geometrically conjugate in $Z_{G_\eta}(\phi_\eta)$. Thus the uniqueness assertion in \cite[Prop.\ 5.4.5]{Conrad} shows that $L_1$ and $L_\eta$ are conjugate by an element of $Z_G(\phi)(k(\overline{\eta}))$, so also $P_{L_1}(\phi_\eta)$ and $P_{L_\eta}(\phi_\eta)$ are $Z_G(\phi)(k(\overline{\eta}))$-conjugate. In particular, $P_{L_\eta}(\phi_\eta)$ is a distinguished parabolic subgroup of $L_\eta$, since the property of being distinguished is preserved by conjugacy. Moreover, since the property of being distinguished involves only locally constant numerical data, we see that $P_{L_s}(\phi_s)$ is also distinguished parabolic in $L_s$, as desired. Similarly, the $2$-weight space for the action of $\phi$ on $\Lie P_L(\phi)$ is nonzero, since this may be checked over $\eta$. \smallskip

Now let $Y$ be a section of the $(\phi, 2)$-weight space of $\Lie P_L(\phi)$ whose special fiber lies in the Richardson orbit, so for dimension reasons every other fiber of $Y$ must also lie in the Richardson orbit of the corresponding fiber of $P_L(\phi)$. For all $t \in \Spec A$, $(L_t, P_{L_t}(\phi_t))$ is a Bala--Carter datum of $Y_t$ and (thus) $\phi_t$ is associated to $Y_t$. In particular, $P = P_G(\phi)$ is the instability parabolic for $Y$ on all fibers.
\end{proof}

\begin{lemma}\label{lemma:mcninch-relative-instability-parabolic}
Let $A$ be a DVR, and let $G$ be a reductive group scheme over $A$ such that the residue characteristic of $A$ is good for $G$. Let $X \in \Lie G$ be a pure fiberwise nilpotent section. Then there exists a parabolic subgroup $P$ of $G$ such that $P_t$ is the instability parabolic of $X_t$ for both $t$. If $\tau: \bG_m \to P$ is a cocharacter such that $\tau_\eta$ is almost associated to $X_\eta$, then also $\tau_s$ is almost associated to $X_s$.
\end{lemma}

\begin{proof}
Let $P$ be the parabolic $A$-subgroup scheme of $G$ whose generic fiber is the instability parabolic of $X_\eta$ as in the paragraph preceding Proposition~\ref{proposition:properties-of-instability-parabolic}; such $P$ exists by properness of the scheme of parabolics. To prove the lemma, it suffices to show that $P_s$ is the instability parabolic of $X_s$, and thus we may extend $A$ at will; in particular, we may assume that $A$ is complete with algebraically closed residue field.\smallskip

Let $T$ be a (split) maximal $A$-torus of $P$ and let $\tau$ be the cocharacter of $T$ such that $\tau_{\eta}$ is almost associated to $X_\eta$, as exists by Proposition~\ref{proposition:mcninch-almost-associated-cocharacters}(\ref{item:almost-associated-cocharacters-exist}) and the fact that $\Hom_A(\bG_m, T) = \Hom_{k(\eta)}(\bG_m, T_\eta)$. By Proposition~\ref{proposition:mcninch-almost-associated-cocharacters}(\ref{item:almost-associated-cocharacters-dynamic-method}), we have $P_\eta = P_{G_\eta}(\tau_{\eta})$ and thus $P = P_G(\tau)$. So by Proposition~\ref{proposition:properties-of-instability-parabolic}(\ref{item:instability-parabolic-open-dense-orbit}), the $\Ad_{P_\eta}$-orbit of $X_\eta$ is open and dense in $(\bigoplus_{n \geq 2} \fg(\tau, n))_\eta$. In particular, $X$ lies in $\bigoplus_{n \geq 2} \fg(\tau, n)$, so also the $\Ad_{P_s}$-orbit of $X_s$ lies in $(\bigoplus_{n \geq 2} \fg(\tau, n))_s$; we claim that it is open and dense in this subscheme of $\fg_s$. \smallskip

Let $d$ be the $A$-rank of $\bigoplus_{n \geq 2} \fg(\tau, n)$, so that
\[
\dim P_\eta - \dim Z_P(X)_\eta = d.
\]
By Proposition~\ref{proposition:properties-of-instability-parabolic}(\ref{item:instability-parabolic-same-centralizer}), we have $Z_P(X)_\eta = Z_G(X)_\eta$, so by purity of $X$ we have
\[
\dim P_s - \dim Z_P(X)_s \geq \dim P_s - \dim Z_G(X)_s = d.
\]
Thus the $\Ad_{P_s}$-orbit of $X_s$ is of dimension $\geq d$. Since it lies in the $d$-dimensional subscheme $(\bigoplus_{n \geq 2} \fg(\tau, n))_s$, it follows that the $\Ad_{P_s}$-orbit of $X_s$ is open and dense in $(\bigoplus_{n \geq 2} \fg(\tau, n))_s$, as desired. \smallskip

Now by Lemma~\ref{lemma:mcninch-existence-of-sections}, after possibly extending $A$ there is some section $Y$ of $\fg$ such that $P_t$ is the instability parabolic of $Y_t$ and $\tau_t$ is associated to $Y_t$ for both $t \in \Spec A$. In particular, the $\Ad_{P_s}$-orbit of $Y_s$ is open and dense in $(\bigoplus_{n \geq 2} \fg(\tau, n))_s$, so because $k(s)$ is algebraically closed we see that $X_s$ and $Y_s$ are $\Ad_{P(k(s))}$-conjugate. Since $P_s$ is the instability parabolic of $Y_s$, it is also the instability parabolic of $X_s$, and similarly $\tau_s$ is almost associated to $X_s$.
\end{proof}

\begin{theorem}\label{theorem:flat-pure-centralizer}
Let $S$ be a reduced scheme and let $G$ be a reductive group scheme over $S$ with Lie algebra $\fg$. Assume that $\chara k(s)$ is good for $G_s$ for every $s \in S$ and $|\pi_1(\sD(G))|$ is invertible on $S$.
\begin{enumerate}
    \item\label{item:pure-Lie-algebra-centralizer} If $X \in \fg(S)$ is pure and the dimension of the centralizer of the semisimple part of $X_s$ is locally constant for $s \in S$, then $Z_G(X)$ is flat and $Z_G(X)/Z(G)$ is smooth.
    \item\label{item:strongly-pure-group-section} If $g \in G(S)$ is strongly pure and the dimension of the centralizer of the semisimple part of $g_s$ is locally constant for $s \in S$, then $Z_G(g)$ is flat and $Z_G(g)/Z(G)$ is smooth.
\end{enumerate}
\end{theorem}

\begin{proof}
By working locally and spreading out (using Lemma~\ref{lemma:locus-of-d-e-elements}), we may and do assume that $S$ is noetherian. By the valuative criterion of flatness \cite[IV\textsubscript{3}, Thm.\ 11.8.1]{EGA}, we may and do assume that $S = \Spec A$ for a DVR $A$. We first establish (\ref{item:pure-Lie-algebra-centralizer}) in the case that $X$ is fiberwise nilpotent. In this case, Lemma~\ref{lemma:mcninch-relative-instability-parabolic} shows that there exists a parabolic subgroup $P$ of $G$ and a cocharacter $\tau: \bG_m \to G$ such that $P_t$ is the instability parabolic of $X_t$ and $\tau_t$ is almost associated to $X_t$ for both $t \in S$. Let $Q = U_G(\tau)$ be the unipotent radical of $P$. \smallskip

We will prove first that $Z_P(X)$ is flat. Let $Y$ be the affine space over $S$ associated to the (free) vector bundle $\bigoplus_{n \geq 2} \fg(\tau, n)$ on $S$, and consider the morphism $f: P \to Y$ given by $f(g) = \Ad(g)X$. Note that $Z_P(X)$ is given by the fiber product
\[
\begin{tikzcd}
Z_P(X) \arrow[r] \arrow[d]
    &P \arrow[d, "f"] \\
S \arrow[r, "X"]
    &Y
\end{tikzcd}
\]
So to prove that $Z_P(X)$ is flat, it suffices to prove that $f$ is flat. For this, we note that $P$ is $S$-flat, so by the fibral flatness criterion \cite[IV\textsubscript{3}, Cor.\ 11.3.11]{EGA} we need only prove flatness between fibers over $S$; i.e., we may and do (temporarily) assume $S = \Spec k$ for a field $k$. By Proposition~\ref{proposition:properties-of-instability-parabolic}(\ref{item:instability-parabolic-open-dense-orbit}) and Proposition~\ref{proposition:mcninch-almost-associated-cocharacters}(\ref{item:almost-associated-cocharacters-agree}), $f$ has open dense image $U \subset Y$; since open embeddings are flat, we may reduce to proving that the map $P \to U$ (which we will also denote by $f$) is flat.\smallskip

First we remark that the natural morphism $Z_P(X) \times P \to P \times_U P$, $(g, p) \mapsto (p, pg)$ is an isomorphism, as may be checked functorially. So we have a Cartesian square
\[
\begin{tikzcd}
Z_P(X) \times P \arrow[r, "\alpha"] \arrow[d, "\rm{pr}_2"]
    &P \arrow[d, "f"] \\
P \arrow[r]
    &U
\end{tikzcd}
\]
where $\alpha: Z_P(X) \times P \to P$ denotes the action map $\alpha(g, p) = pg$. Since the base is now a field, the projection map $Z_P(X) \times P \to P$ is flat, and thus in particular all fibers are the same dimension. This property is insensitive to surjective base change, so all fibers of the morphism $f: P \to U$ are also of the same dimension. Since $U$ is regular, Miracle Flatness \cite[Thm.\ 23.1]{Matsumura} shows that $f$ is therefore flat. Thus (again over a general base) $Z_P(X)$ is flat. \smallskip

Now return to the case that $S$ is the spectrum of a DVR and consider the natural inclusion $i: Z_P(X) \to Z_G(X)$. We claim that this is an isomorphism, which will prove that $Z_G(X)$ is flat. But we know that $Z_P(X)$ is flat, so by the fibral isomorphism criterion \cite[IV\textsubscript{4}, Cor.\ 17.9.5]{EGA} we may pass to the case that $S = \Spec k$ for a field $k$. In this case, Proposition~\ref{proposition:properties-of-instability-parabolic}(\ref{item:instability-parabolic-same-centralizer}) shows that indeed $i$ is an isomorphism. \smallskip

To prove that $Z_G(X)/Z(G)$ (which is flat because it is a quotient of the flat $S$-scheme $Z_P(X) = Z_G(X)$ by a flat $S$-group scheme) is smooth, it suffices by flatness to work fibrally, whence we may assume again that $S$ is the spectrum of a field. But now this follows from Corollary~\ref{corollary:smooth-centralizer-pretty-good-characteristic}. This establishes (\ref{item:pure-Lie-algebra-centralizer}) in the fiberwise nilpotent case. \smallskip

Now consider (\ref{item:pure-Lie-algebra-centralizer}) without the assumption of fiberwise nilpotence. By assumption and Theorem~\ref{theorem:relative-jordan-decomposition-over-normal-base}, we may extend $A$ to assume that there exist $X_{\rm{ss}}, X_{\rm{n}} \in \fg$ such that $X = X_{\rm{ss}} + X_{\rm{n}}$ induces the Jordan decomposition on both fibers. By Lemma~\ref{lemma:semisimple-centralizer-is-reductive}, $Z_G(X_{\rm{ss}})$ is a reductive group scheme over $A$. By Lemma~\ref{lemma:semisimple-replacement} and \cite[Cor.\ 3.2]{Integral-Springer}, $Z_G(X_{\rm{ss}})$ satisfies the hypotheses of the theorem. Thus we may replace $G$ by $Z_G(X_{\rm{ss}})$ to reduce to the already-treated fiberwise nilpotent case. \smallskip

Finally we consider (\ref{item:strongly-pure-group-section}). By assumption and Theorem~\ref{theorem:relative-jordan-decomposition-over-normal-base}, we may extend $A$ to assume that there exist $t, u \in G(A)$ such that $g = tu$ is the Jordan decomposition on both fibers. Moreover, since $g$ is strongly pure, it follows that $t$ is also strongly pure: indeed, $\dim Z_G(t)_s = \dim Z_G(t)_\eta$ by assumption, $t$ is fiberwise ordinary essentially by definition (Definition~\ref{definition:ordinary}). Thus Lemma~\ref{lemma:semisimple-centralizer-is-reductive} implies that $Z_G(t)$ is a reductive group scheme over $A$. By \cite[Cor.\ 3.2]{Integral-Springer}, $Z_G(t)$ satisfies the hypotheses of the theorem, so we may replace $G$ by $Z_G(t)$ and $g$ by $t^{-1}g$ to assume that $g$ is fiberwise unipotent. \smallskip

Under our assumptions, \cite[Thm.\ 5.1]{Integral-Springer} shows that there is a $G$-equivariant isomorphism $\phi: \sU_G \to \sN_G$ from the unipotent scheme to the nilpotent scheme of $G$ (each obtained by base change of the corresponding object over $\bZ$ as defined there). Let $X = \phi(g)$, so that $X$ is pure and fiberwise nilpotent, $Z_G(g) = Z_G(X)$, and $Z_G(g)/Z(G) = Z_G(X)/Z(G)$. Thus the claims follow from the already-established result for pure fiberwise nilpotent elements.
\end{proof}

In type A, one can show that the flatness assertion of Theorem~\ref{theorem:flat-pure-centralizer} holds when the assumption on $X$ is weakened simply to strong purity; one uses principally the fact that the centralizer of any element of $\GL_n$ over a field is connected. In general, we have the following result.

\begin{lemma}\label{corollary:flat-pure-centralizer-small-extension}
Let $A$ be a DVR, and let $G$ be a reductive $A$-group scheme. Let $P$ be a parabolic $S$-subgroup scheme of $G$, let $U_P$ be the unipotent radical of $P$, and let $L$ be a Levi subgroup of $P$. Let $u \in P(A)$ be a pure fiberwise unipotent section such that the $P_s$-orbit of $u_s$ is dense in $(U_P)_s$ (i.e., $u_s$ lies in the Richardson orbit), and let $t \in Z(L)(A)$ be such that $t_s = 1$.
\begin{enumerate}
    \item\label{item:flatness-extension-group-case} If $Z_{G_s}(u_s) \subset P_s$, then $Z_G(tu)$ is flat and $Z_G(tu)/Z(G)$ is smooth.
    \item\label{flatness-failure-group-case} If $Z_{G_s}(u_s) \not\subset P_s$, then there is some $t$ as above such that $Z_G(tu)$ is not flat, although $tu$ is strongly pure.
    %\item\label{item:flatness-extension-lie-algebra-case} Suppose that the $\Ad_{P_s}$-orbit of $X_s$ is dense in $\Lie (U_P)_s$, and let $Y \in \Lie Z(L_0)$ be such that $Y_s = 0$. Then $Z_G(X + Y)$ is flat and $Z_G(X + Y)/Z(G)$ is smooth.
\end{enumerate}
Completely analogous assertions hold for section of $\Lie G$ in place of sections of $G$.
\end{lemma}

\begin{proof}
We first deal with (\ref{item:flatness-extension-group-case}). By flat descent, we may assume that $A$ is complete with algebraically closed residue field. By embedding $Z(G)$ into a torus $T$ and replacing $G$ by $G \times^{Z(G)} T$ as in the proof of Proposition~\ref{prop:scheme-of-d-regular-elements}, we may assume that $Z(G)$ is smooth, in which case we will show that $Z_G(tu)$ is smooth. The local criterion of flatness \cite[Thm.\ 22.1]{Matsumura} reduces us to showing that for each integer $n \geq 1$ there is some section $t_n \in Z(L)(A)$ agreeing with $t$ in $Z(L)(A/\pi^n)$ such that $Z_G(t_n u)$ is flat and $Z_G(t_n u)/Z(G)$ is smooth.\smallskip

There is a nonempty Zariski open subset $V$ of $Z(L_\eta)$ such that for all $t_0 \in V(k(\overline{\eta}))$ we have $Z_{G_\eta}(t_0) = L_\eta$. Being nonempty and Zariski dense, $V(k(\eta))$ is also open and dense in the valuation topology of $Z(L)(k(\eta))$. Thus there are sections $t' \in Z(L)(A)$ arbitrarily close to $t$ in the valuation topology such that $Z_{G_\eta}(t'_\eta) = L_\eta$, and by the previous paragraph we may and do assume that this is the case for $t' = t$. Applying \cite[Lem.\ 2.12]{SteinbergReg} to the connected solvable group $Z(L_\eta)^0_{\rm{red}} \ltimes (U_P)_\eta$, after passing to a ramified extension of $A$ we may assume that $t_\eta u_\eta$ is conjugate to $t_\eta u'$ for some $u' \in U_P(k(\eta))$ commuting with $t_\eta$. Since $L \cap U_P = \{1\}$, we see that $u' = 1$ and thus $Z_{G_\eta}(t_\eta u_\eta)$ is conjugate to $L_\eta$. In particular, $\dim Z_{G_\eta}(t_\eta u_\eta) = \dim L_\eta$ and $tu$ is strongly pure. \smallskip

We will now show that $Z_P(tu)$ is flat under the assumption that $Z_{G_s}(u_s) \subset P_s$. From this, using the fibral isomorphism criterion we may conclude that the map $Z_P(tu) \to Z_G(tu)$ is an isomorphism, so also $Z_G(tu)$ is flat. Now consider the morphism $\Phi: P \to tU_P$ given by $\Phi(g) = gtug^{-1}$. To check that $Z_P(tu)$ is flat, it suffices to check that $\Phi$ is flat. Since $P$ and $tU_P$ are both smooth, flatness of $\Phi$ may be checked on fibers over $A$, where it follows from the fact that $\Phi$ is a $Z_P(tu)$-torsor onto a dense open subscheme of $tU_P$. \smallskip

In the second case, choose $t$ such that $Z_{G_\eta}(t_\eta) = L$, as above. Thus $tu$ is strongly pure and $Z_{G_\eta}(t_\eta u_\eta) \subset P_\eta$. Now note simply that if $Z_G(tu)$ were flat then it would be the closure of its generic fiber, so in particular we would have $Z_G(tu) \subset P$. Since this is false by hypothesis, $Z_G(tu)$ is not flat.
\end{proof}

\begin{remark}\label{remark:possible-extensions-of-flatness-results}
A nilpotent element $X$ of the Lie algebra of a reductive group $G$ is called \textit{even} if $X$ is a Richardson element of its instability parabolic. By using Corollary~\ref{corollary:flat-pure-centralizer-small-extension} (and Proposition~\ref{proposition:properties-of-instability-parabolic}(\ref{item:instability-parabolic-same-centralizer})), one can interpret the centralizer of an even nilpotent element as a degeneration of a Levi of its instability parabolic (in good characteristic $p$ with $p \nmid |\pi_1(\sD(G))|$), fulfilling a desire stated in the introduction of \cite{McNinch-Testerman} (and also providing an alternate proof of the main result of that paper). An interesting question is whether $Z(Z_G(g))$ is flat whenever $Z_G(g)$ is flat; an affirmative answer to this question would lead to an improvement to the main result of \cite{McNinch-Testerman} (and hence a case-free argument for one of the main results of \cite{Lawther-Testerman}). \smallskip

It is always true that if $G$ is a reductive group over a field of pretty good characteristic and $u \in G$ is Richardson in a parabolic subgroup $P$ of $G$ then $Z_G(u)^0 \subset P$ (see, e.g., \cite[4.9]{Jantzen-nilpotent}). However, it may happen that $Z_G(u) \not\subset P$, even over a field of characteristic $0$. As an explicit example, let $G = \Sp_4$ over a field $k$ of arbitrary characteristic and consider the element
\[
g = \begin{pmatrix}
1 &0 &0 &1 \\
0 &1 &1 &0 \\
0 &0 &1 &0 \\
0 &0 &0 &1
\end{pmatrix}
\]
which is Richardson in the parabolic subgroup $P$ of $G$ which one writes as
\[
P = \left\{\begin{pmatrix}
* &* &* &* \\
0 &* &* &* \\
0 &0 &* &0 \\
0 &* &* &*
\end{pmatrix}\right\}.
\]
However, notice that $g$ is centralized by the matrix
\[
\begin{pmatrix}
0 &1 &0 &0 \\
1 &0 &0 &0 \\
0 &0 &0 &1 \\
0 &0 &1 &0
\end{pmatrix}
\]
which does not lie in $P$. Thus by Lemma~\ref{corollary:flat-pure-centralizer-small-extension}, there exists some strongly pure section $g \in P(k[\![t]\!])$ such that $g_s = u$ and such that $g$ has non-flat centralizer. For more examples, see \cite{Hesselink}; in the terminology of that paper, $Z_G(u) \subset P$ if and only if $P$ is a ``stable polarization" of $u$.

%In bad characteristic, there are a few basic obstacles to flatness. First, there is no Springer isomorphism, so the linear algebraic technique of associated cocharacters does not apply in the case of fiberwise unipotent sections. Second, associated cocharacters do not always exist for nilpotent elements. For a similar reason, there are ``non-standard" unipotent and nilpotent conjugacy classes for some semisimple groups in small characteristic; i.e., there are conjugacy classes which cannot be described as in the Bala--Carter--Pommerening theorem. For more details on this point, see \cite[\S 5.11]{Carter}.\smallskip

%Finally, nilpotent and unipotent centralizers need not be smooth modulo the center: in the nilpotent case, this can be seen from the calculations in \cite[\S 2]{Springer}, and we will see in Section~\ref{section:examples} that this can cause flatness to fail. In the unipotent case, one can see this failure of smoothness by comparing the tables in \cite{Liebeck-Seitz} (computing the dimension of $Z_G(u)$) with those in \cite{Lawther} (computing the number and sizes of Jordan blocks in the action of $u$ on $\Lie G$, and hence the dimension of $(\Lie G)^{\Ad u}$). It would not be surprising if flatness can fail in mixed characteristic when some fiber of $g$ is non-standard unipotent with non-smooth centralizer. In Section~\ref{subsection:regular-centralizers}, we will show nonetheless that centralizers of \textit{fiberwise regular} sections of $G$ are always flat under the sole assumption that $|\pi_1(\sD(G))|$ is invertible on the base.
\end{remark}

We conclude this section by showing the existence of a flattening stratification for the universal centralizer of $G$, consisting of $S$-flat pieces. In view of the preceding remark, it is not enough to stratify $G$ by the dimensions of the stabilizers. Recall the subsets $G_{d, e}$ defined in Section~\ref{section:constructibility-nonsense}; by Corollary~\ref{cor:openness-of-d-e-locus}, each is locally closed in $G$. Recall also the subset $G_{\rm{ord}}$ of $G$, which is open in $G$ under some mild hypotheses; see Proposition~\ref{prop:strong-purity-generizes}.

\begin{cor}\label{cor:flat-strat}
Let $G$ be a reductive group scheme over a Dedekind domain $R$ such that $|\pi_1(\sD(G))|$ is invertible in $R$, $Z(G)$ is smooth, and every residue characteristic of $R$ is good for $G$. Equip $G_{d, e\rm{-ord}}$ with the reduced subscheme structure. Then each $G_{d, e\rm{-ord}}$ is $R$-flat and the subschemes $G_{d, e\rm{-ord}}$ of $G_{\rm{ord}}$ form a flattening stratification for the centralizer of the universal point over $G_{\rm{ord}}$.
\end{cor}

\begin{proof}
By Theorem~\ref{theorem:flat-pure-centralizer}, the only thing to prove is that each $G_{d, e\rm{-ord}}$ is $R$-flat. For this we may pass from $R$ to a local ring $R_{\fp}$ and thus assume that the base is the spectrum of a DVR. Since $R$ is Dedekind, it suffices to show that $G_{d, e\rm{-ord}}$ is the schematic closure of its generic fiber, and by reducedness it suffices to show that it is the \textit{topological} closure of its generic fiber. Thus we are only required to show that any point in the special fiber of $G_{d, e\rm{-ord}}$ has a generization in the generic fiber. Let $g \in G_{d, e}$ lie in the special fiber, and let $A$ be a complete DVR over $R$ with residue field $\overline{k(g)}$. Let $t_0$ be the semisimple part of $g_{\overline{k(g)}}$; since $A$ is complete, Proposition~\ref{prop:scheme-of-d-regular-elements} shows that there is a strongly pure section $t \in G(A)$ with special fiber $t_0$. We may pass from $G$ to $Z_{G_A}(t)$ and from $g$ to $t_s^{-1}g$ to assume that $g$ is a rational point of the special fiber and it is unipotent. \smallskip

Fix a Springer isomorphism $\rho: \sU_G \to \sN_G$, which exists by \cite[Thm.\ 5.1]{Integral-Springer}, and let $X = \rho(g)$, so $X$ is a nilpotent element of $\fg \otimes_A k(s)$. By Lemma~\ref{lemma:existence-of-associated-cocharacters}, there exists a cocharacter $\tau: \bG_m \to G_s$ associated to $X$. If $P = P_{G_s}(\tau)$ is the corresponding parabolic subgroup of $G_s$ then $X$ is a rational point of the unipotent radical $U = U_{G_s}(\tau)$, and by Proposition~\ref{proposition:properties-of-instability-parabolic} the $\Ad_{P_s}$-orbit of $X$ is open and dense in $\bigoplus_{n \geq 2} \fg_s(\tau, n)$. Lift $\tau$ to a cocharacter $\widetilde{\tau}: \bG_m \to G$, and let $\widetilde{P}$ be the corresponding parabolic. We may lift $X$ to an element $\widetilde{X} \in \fg(\widetilde{\tau}, 2)$, so that also $\widetilde{\tau}_\eta$ is an associated cocharacter for $\widetilde{X}_\eta$. Thus by Lemma~\ref{lemma:mcninch-relative-instability-parabolic} we see that the $\Ad_{\widetilde{P}_{\eta}}$-orbit of $\widetilde{X}_{\eta}$ is open and dense in $\bigoplus_{n \geq 2} \fg_\eta(\widetilde{\tau}_\eta, n)$. Since $\fg(\widetilde{\tau}, n)$ is a free $A$-module, we see that $\dim Z_{P_s}(X) = \dim Z_{\widetilde{P}_\eta}(X_\eta)$. If $\widetilde{g} = \rho^{-1}(\widetilde{X})$, then $\widetilde{g}$ is fiberwise unipotent and pure (by $G$-equivariance of $\rho$), so we are done.
\end{proof}

\subsection{The structure of pure centralizers}

Having established the desired flatness result, we move on to considering some corollaries. Our main goal is to establish structural results concerning $Z_G(X)$ for a pure fiberwise nilpotent section $X$ of $\fg$; specifically, we will show under mild hypotheses on $S$ and $G$ that there are short exact sequences
\begin{gather*}
1 \to R_X \to Z_G(X) \to C_X \to 1 \\
1 \to Z(G) \to C_X \to D_X \to 1 \\
1 \to D_X^0 \to D_X \to D_X/D_X^0 \to 1
\end{gather*}
in which $R_X$ is smooth with connected unipotent fibers, $D_X^0$ is a reductive group scheme, and $D_X/D_X^0$ is finite etale (and thus in particular $C_X$ is finitely presented and affine). The first sequence splits etale-locally on $S$. Using some cohomological calculations, including certain instances of the \textit{Grothendieck--Serre conjecture}, we will show in Theorem~\ref{theorem:grothendieck-serre} that if $S = \Spec A$ then under certain hypotheses on the base ring $A$, two pure fiberwise nilpotent sections of $\fg$ are $\Ad_{G(A)}$-conjugate if and only if they are $\Ad_{G(k(\eta))}$-conjugate. Specifically, we will show that if $Y \in \fg$ is another pure fiberwise nilpotent section which is generically (rationally) conjugate to $X$, then $\Transp_G(X, Y)$ is a $Z_G(X)$-torsor which is trivial over $\eta$, and moreover any such generically trivial $Z_G(X)$-torsor is trivial. \smallskip

Before establishing any of these corollaries, we will find it necessary to establish more properties of the instability parabolic and associated cocharacters over general reduced bases. In particular, defining $R_X$, showing that $R_X$ is smooth, and showing that $Z_G(X)/R_X$ exists as a (flat) scheme will all require having fiberwise (almost) associated cocharacters over some base. Specifically, if $\tau: \bG_m \to G$ is fiberwise almost associated to $X$, then $R_X$ will be defined to be $Z_G(X) \cap U_G(\tau)$. By Proposition~\ref{proposition:properties-of-instability-parabolic}, the fibers of this $R_X$ are independent of the fibers of $\tau$, so $R_X$ will be seen to be independent of $\tau$ as soon as it is seen to be flat. \smallskip

To show that $R_X$ is flat, we will construct a ``Levi factor" $L_X$ in $Z_G(X)$ etale-locally on the base. The term ``Levi factor" is in quotation marks because $L_X$ is not generally smooth. To be more specific, if $\tau: \bG_m \to G$ is fiberwise associated to $X$, then we will define $L_X = Z_G(X) \cap Z_G(\tau)$ and show that $L_X \ltimes R_X \to Z_G(X)$ is an isomorphism of group schemes. In particular, $L_X$ is flat, and $L_X/Z(G)$ and $R_X$ are smooth. Moreover, $L_X/Z(G)$ has reductive fibers, whence the somewhat imprecise term ``Levi factor". We warn that despite the notation, $L_X$ is highly dependent on the choice of $\tau$. \smallskip

Once we have established this semidirect product decomposition, the displayed short exact sequences above will pop out naturally. In light of the above discussion, we see that the following lemma will be instrumental to understanding the structure of $Z_G(X)$.

\begin{lemma}\label{lemma:flat-pure-centralizer-consequences}
Let $S$ be a reduced locally noetherian scheme and let $G$ be a reductive $S$-group scheme such that for each $s \in S$, $\chara k(s)$ is good for $G_s$. Let $X \in \fg(S)$ be a pure fiberwise nilpotent section. 
\begin{enumerate}
    \item\label{item:instability-parabolic-exists-integrally} If $S$ is normal then there exists a unique parabolic $S$-subgroup $P$ of $G$ such that $P_s$ is the instability parabolic of $X_s$ for all $s \in S$. Moreover, Zariski-locally on $S$ there exists a cocharacter $\tau: \bG_m \to P$ such that $\tau_s$ is almost associated to $X_s$ for all $s \in S$.
    \item\label{item:orbit-map-is-flat} Suppose that $|\pi_1(\sD(G))|$ is invertible on $S$. If $\tau: \bG_m \to G$ is a cocharacter such that $\tau_s$ is almost associated to $X_s$ for all $s \in S$, then the orbit map $\Phi: P_G(\tau) \to \bigoplus_{n \geq 2} \fg(\tau, n)$, $\Phi(g) = (\Ad g) X$, is flat. If $V$ is the (open) image of $\Phi$, then $\Phi: P_G(\tau) \to V$ is a $Z_G(X)$-torsor.
    \item\label{item:associated-characters-exist-integrally} If $S = \Spec A$ is strictly henselian local, then there exists a cocharacter $\tau: \bG_m \to G$ such that $\tau_s$ is associated to $X_s$ for all $s$.
\end{enumerate} 
\end{lemma}

\begin{proof}
First, if $\eta$ is a generic point of $S$ and $\Spec A \to S$ is a map from the spectrum of a DVR which maps the generic point of $\Spec A$ to $\eta$, then there is a unique parabolic $A$-subgroup $P_A$ of $G_A$ whose generic fiber is pulled back from the instability parabolic $P_\eta$ of $X_\eta$, and Lemma~\ref{lemma:mcninch-relative-instability-parabolic} shows that the special fiber of $P_A$ is the instability parabolic of the corresponding fiber of $X$. By \cite[2.6.3]{McNinch-relative-centralizer}, since this is true for every such map $\Spec A \to S$ and the formation of the instability parabolic commutes with arbitrary field extensions, the first claim of (\ref{item:instability-parabolic-exists-integrally}) follows. To establish the cocharacter claim, we may work Zariski-locally and thus by \cite[Exp.\ XIV, Cor.\ 3.20, Rmk.\ 3.21]{SGA3II} we may assume that there exists a maximal $S$-torus $T$ of $P$. Note that $\Hom_A(\bG_m, T)$ surjects onto $\Hom_{k(\eta)}(\bG_m, T_\eta)$, so by Proposition~\ref{proposition:mcninch-almost-associated-cocharacters} there is some cocharacter $\tau$ of $T$ such that $\tau_\eta$ is almost associated to $X_\eta$. Then Lemma~\ref{lemma:mcninch-relative-instability-parabolic} shows that $\tau_s$ is almost associated to $X_s$ for all $s \in S$. \smallskip

Next we want to establish (\ref{item:orbit-map-is-flat}). By the fibral flatness criterion, it suffices to establish flatness of $\Phi$ when the base is the spectrum of a field, where it follows from the fact that $\Phi$ is a $Z_P(X) = Z_G(X)$-torsor over its open dense image. Over a general base, the fact that $\Phi$ is flat and finitely presented implies that its image $V$ is open, and it is evidently dense on fibers. The map $\Phi: P \to V$ is now fppf, and since $P \times_V P \cong Z_G(X) \times P$ we find that $\Phi$ is a $Z_G(X)$-torsor, as desired. \smallskip

Now we turn to proving (\ref{item:associated-characters-exist-integrally}). Fix a cocharacter $\phi: \bG_m \to P$ such that $\tau_\eta$ is almost associated to $X_\eta$. By Lemma~\ref{lemma:mcninch-existence-of-sections} (whose hypotheses hold by (\ref{item:instability-parabolic-exists-integrally})), there exists a section $Y \in \fg$ such that $Y_\eta$ is geometrically conjugate to $X_\eta$ and $\phi_t$ is associated to $Y_t$ for all $t \in S$. We claim that $X$ and $Y$ are $\Ad_{G(A)}$-conjugate; for this, we may use pushouts as usual to arrange that $Z(G)$ is smooth, and in particular that $Z_G(X)$ is smooth by Theorem~\ref{theorem:flat-pure-centralizer}. But now (\ref{item:orbit-map-is-flat}) shows that $\Transp_G(X, Y)$ is a $Z_G(X)$-torsor, so it has an $A$-point since $A$ is strictly henselian. Thus $X$ and $Y$ are $\Ad_{G(A)}$-conjugate, so the desired $\tau$ is a $G(A)$-conjugate of $\phi$.
\end{proof}

If $G$ is a reductive $S$-group scheme and $X \in \fg(S)$ is a pure nilpotent section, then one can define the normalizer scheme $N_G(X)$ as in the displayed equation (\ref{equation:normalizer-scheme}). The main point of $N_G(X)$ for our purposes is that for reduced $S$: $N_G(X)$ is flat, it contains $Z_G(X)$, and any fiberwise associated cocharacter of $X$ factors through $N_G(X)$. Of these points, only flatness is non-trivial, and it will be the first point established below.

\begin{theorem}\label{theorem:flat-pure-centralizer-semidirect-product}
Let $S$ be a reduced scheme and let $G$ be a reductive $S$-group scheme such that for each $s \in S$, $\chara k(s)$ is good for $G_s$ and $\chara k(s) \nmid |\pi_1(\sD(G))|$. Let $X \in \fg(S)$ be a pure fiberwise nilpotent section.
\begin{enumerate}
    \item\label{item:relative-normalizer-scheme} The normalizer scheme $N_G(X)$ is $S$-flat and $N_G(X)/Z(G)$ is $S$-smooth.
    \item\label{item:nilpotent-centralizer-unipotent-part} Suppose that $\tau: \bG_m \to G$ is a cocharacter such that $\tau_s$ is almost associated to $X_s$ for all $s \in S$, and let $R_X = Z_{U_G(\tau)}(X)$. Then $R_X$ is $S$-smooth with connected unipotent fibers and the fppf quotient sheaf $C_X = Z_G(X)/R_X$ is representable by a flat $S$-group scheme and $D_X := C_X/Z(G)$ is smooth with reductive fibers. In particular, the relative identity component $D_X^0$ \cite[IV\textsubscript{3}, Cor.\ 15.6.5]{EGA} is a reductive $S$-group scheme.
    \item\label{item:nilpotent-centralizer-semidirect-product} Suppose that $\tau: \bG_m \to G$ is a cocharacter such that $\tau_s$ is associated to $X_s$ for all $s \in S$. Then $L_X := Z_G(\tau) \cap Z_G(X)$ is $S$-flat, and $L_X/Z(G)$ is $S$-smooth. If $Z(G)$ is smooth then the relative identity component $L_X^0$ is a reductive $S$-group scheme (so $L_X$ is smooth). The multiplication morphism
    \[
    L_X \ltimes R_X \to Z_G(X)
    \]
    is an isomorphism of $S$-group schemes.
    \item\label{item:finite-etale-component-group} The fppf quotient sheaf $D_X/D_X^0$ is representable by a finite etale $S$-group scheme of order invertible on $S$. If $Z(G)$ is smooth, then $C_X/C_X^0$ is representable by a finite etale $S$-group scheme of order invertible on $S$.
\end{enumerate}
\end{theorem}

\begin{proof}
For (\ref{item:relative-normalizer-scheme}), note that as in Lemma~\ref{lemma:normalizer-scheme} there is a short exact sequence
\[
1 \to Z_G(X) \to N_G(X) \to \bG_m \to 1.
\]
Thus the flatness and smoothness assertions reduce to the analogous assertions for $Z_G(X)$, already proved in Theorem~\ref{theorem:flat-pure-centralizer}. \smallskip

For the first smoothness claim in (\ref{item:nilpotent-centralizer-unipotent-part}), we may use flat descent to assume that $S = \Spec A$ for a strictly henselian ring $A$. By Proposition~\ref{proposition:properties-of-instability-parabolic}(\ref{item:instability-parabolic-semidirect-product}), $R_X$ has smooth connected unipotent fibers, so it suffices to show that $R_X$ is flat. By Proposition~\ref{proposition:mcninch-almost-associated-cocharacters}(\ref{item:almost-associated-cocharacters-dynamic-method}), $R_X$ does not change if $\tau$ is replaced by another cocharacter which is fiberwise almost associated to $X$. Thus by Lemma~\ref{lemma:flat-pure-centralizer-consequences}(\ref{item:associated-characters-exist-integrally}) we may assume that $\tau_t$ is \textit{associated} to $X_t$ for all $t \in \Spec A$. Thus we may reduce to proving that the multiplication morphism in (\ref{item:nilpotent-centralizer-semidirect-product}) is an isomorphism. \smallskip

To show that $C_X$ is representable by a flat $S$-group scheme, effectivity of fpqc descent for affine morphisms \cite[Exp.\ VIII, Cor.\ 7.9]{SGA1} allows us to pass to any etale cover of $S$. By Lemma~\ref{lemma:flat-pure-centralizer-consequences}(\ref{item:associated-characters-exist-integrally}) we may assume that there exists a cocharacter $\tau: \bG_m \to G$ such that $\tau_s$ is associated to $X_s$ for all $s \in S$. Thus representability again reduces to the claim about the multiplication morphism in (\ref{item:nilpotent-centralizer-semidirect-product}); the flatness and smoothness assertions also reduce to this, so we will prove it next. \smallskip

For (\ref{item:nilpotent-centralizer-semidirect-product}) we may as usual pass from $G$ to $G \times^{Z(G)} T$ for an $S$-torus $T$ into which $Z(G)$ embeds to assume that $Z(G)$ is $S$-smooth. To prove the flatness and smoothness claims, we may use the valuative criterion of flatness \cite[IV\textsubscript{3}, Thm.\ 11.8.1]{EGA} and flat descent to assume that $S = \Spec A$ for a strictly henselian DVR $A$. The multiplication morphism described in this point is an isomorphism on fibers by Proposition~\ref{proposition:properties-of-instability-parabolic}(\ref{item:instability-parabolic-semidirect-product}), so because $Z_G(X)$ is flat, upper semicontinuity of fiber dimension implies that the fiber dimensions of $L_X$ and $R_X$ are locally constant on $S$. By Proposition~\ref{proposition:properties-of-instability-parabolic}(\ref{item:instability-parabolic-semidirect-product}), $R_X$ has smooth connected unipotent fibers, so Lemma~\ref{lemma:boohers-lemma} implies that $R_X$ is $S$-flat, and thus it is $S$-smooth by fibral considerations. \smallskip

To show that $L_X$ is $S$-flat, it suffices as usual to assume that $Z(G)$ is smooth, in which case we must show that $L_X$ is $S$-smooth. By Lemma~\ref{lemma:flat-pure-centralizer-consequences}(\ref{item:associated-characters-exist-integrally}) and henselianity of $A$, there is a cocharacter $\tau: \bG_m \to G$ such that $\tau_t$ is associated to $X_t$ for all $t \in S$. Since $N_G(X)$ is smooth by (\ref{item:relative-normalizer-scheme}), the centralizer $Z_{N_G(X)}(\phi)$ is also smooth by \cite[Lem.\ 2.2.4]{Conrad}. There is a short exact sequence
\[
1 \to L_X \to Z_{N_G(X)}(\phi) \to \bG_m \to 1,
\]
and the map $Z_{N_G(X)}(\phi) \to \bG_m$ is flat: since $Z_{N_G(X)}(\phi)$ is smooth, this flatness may be checked on fibers, where it holds since precomposing with $\phi$ gives rise to a surjective map by Lemma~\ref{lemma:normalizer-scheme}. So the kernel $L_X$ is $S$-affine and $S$-flat. \smallskip

Now that $L_X$ and $R_X$ are $S$-flat, to check that the multiplication morphism $L_X \ltimes R_X \to Z_G(X)$ is an isomorphism it suffices to check on fibers, where it holds by Proposition~\ref{proposition:properties-of-instability-parabolic}(\ref{item:instability-parabolic-semidirect-product}). \smallskip

Continue to assume that $Z(G)$ is smooth, so $L_X$ is smooth. To show that $L_X^0$ is reductive and $L_X/L_X^0$ is finite, we may and do assume that $S = \Spec A$ for a DVR $A$. We will adapt the proof of \cite[Thm.\ 2.3(iii)]{Premet} using results from \cite{Alper-adequate}. Specifically, since $Z_G(\tau)$ is a reductive group scheme, \cite[Thms.\ 9.4.1 and 9.7.6]{Alper-adequate} show that it suffices to show that the quotient $Z_G(\tau)/L_X$ is \textit{affine}. \smallskip

Note that there is an action map $\alpha: Z_G(\tau) \to \fg(\tau, 2)$ given by $g \mapsto \Ad(g)X$, and the fiber over $X$ is equal to $L_X$. The proof of \cite[Thm.\ 2.3]{Premet} shows that there is an element $Q \in \Sym^*(\fg(\tau, 2)^*)(A)$ such that the image of $\alpha$ is equal to $\fg(\tau, 2) - V(Q)$ on both fibers. But now $W \coloneqq \fg(\tau, 2) - V(Q)$ is an open subscheme of $\fg(\tau, 2)$, so it follows that $Z_G(\tau)/L_X \cong W$: first, a theorem of Artin \cite[Cor.\ 6.4]{Artin-stacks} shows that $Z_G(\tau)/L_X$ is a (separated) algebraic space. By a criterion of Knutson \cite[II, Thm.\ 6.15]{Knutson-algebraic-spaces}, since there is a monomorphism $Z_G(\tau)/L_X \to \fg(\tau, 2)$, automatically $Z_G(\tau)/L_X$ is a scheme. To check that the natural map $\alpha Z_G(\tau)/L_X \to W$ is an isomorphism, \cite[IV\textsubscript{4}, Thm.\ 17.9.1]{EGA} shows that it suffices to check that it is a surjective smooth monomorphism. Clearly $\alpha$ is a monomorphism, and we have already mentioned that it is surjective. To check smoothness, we may pass to fibers to assume $S = \Spec k$ for a field $k$. Note that $W$ and $Z_G(\tau)$ are both smooth over $k$, so in particular they are both regular schemes. Thus by Miracle Flatness \cite[Thm.\ 23.1]{Matsumura}, it suffices to check that all fibers of $\alpha$ are smooth of the same dimension. But in fact all fibers are isomorphic to $L_X$, which we have already seen to be smooth. Thus finally we have $Z_G(\tau)/L_X \cong W$. Since $W$ is affine, \cite[Thms.\ 9.4.1 and 9.7.6]{Alper-adequate} show that $L_X^0$ is reductive and $L_X/L_X^0$ is finite. \smallskip

Finally, we prove (\ref{item:finite-etale-component-group}). Embedding $Z(G)$ into a torus $T$ and passing to $G \times^{Z(G)} T$, we may assume that $Z(G)$ is smooth with connected fibers, so in particular $C_X$ is smooth, as is $Z_G(X)$, so $C_X^0$ and $Z_G(X)^0$ make sense and $Z(G) \subset C_X^0$. Thus $D_X/D_X^0 \cong C_X/C_X^0$, so we may assume that $Z(G)$ is smooth and thereby reduce to the second claim of (\ref{item:finite-etale-component-group}). Also $C_X \cong L_X$, so by the previous paragraph the only remaining thing to show is that $L_X/L_X^0$ is of order invertible on $S$. \smallskip

For this final claim, we may pass to a fiber to assume that $S = \Spec k$ for an algebraically closed field $k$ of characteristic $p > 0$. If $L_X/L_X^0$ has $p$-torsion, then a simple argument shows that there is some unipotent element $u_0 \in L_X(k) - L_X^0(k)$. Using \cite[Thm.\ 5.1]{Integral-Springer}, we fix a $G$-equivariant isomorphism $\rho: \sU_G \to \sN_G$ and let $X_0 = \rho(u_0)$, $u = \rho^{-1}(X)$. Since $u_0$ centralizes $X$, it follows that $u$ centralizes $X_0$. Thus also $u$ centralizes the line $kX_0 \subset \sN_G$, and again it follows that the line $\rho^{-1}(kX_0)$ centralizes $X$. Since this line passes through $u_0$ and $1$, this contradicts the assumption $u_0 \not\in L_X^0(k)$. This completes the proof.
\end{proof}

We mention the following curious corollary. This follows from examining the tables in \cite{Alek79} or \cite{Liebeck-Seitz} (which give far sharper information), so the merit of this result is that its proof is ``conceptual", in that it does not rely on any case-by-case calculations of the component groups of centralizers.

\begin{cor}\label{cor:unip-comp-gp}
Let $k$ be a field, and let $G$ be a connected reductive $k$-group such that $\chara k$ is good for $G$. If $u \in G(k)$ is unipotent, then the only primes dividing the order of $Z_G(u)/Z_G(u)^0Z(G)$ are bad for $G$. A similar claim holds if $X \in \Lie G$ is nilpotent. In particular, the order of $Z_G(u)/Z_G(u)^0Z(G)$ always divides a power of $30$.
\end{cor}

\begin{proof}
We may and do pass from $G$ to the universal cover $\sD(G)$ to assume that $G$ is semisimple and simply connected. We may and do further assume that $k$ is algebraically closed, so $G$ is split and hence there is a split reductive $\bZ$-group scheme $\bG$ such that $G \cong \bG_k$. First suppose $\chara k = p > 0$. In this case, Corollary~\ref{cor:flat-strat} shows that there is a finite extension $\sO$ of $\bZ_p$ and a section $u_0 \in \bG(\sO)$ such that $u$ is conjugate to the special fiber of $u_0$. By Theorem~\ref{theorem:flat-pure-centralizer-semidirect-product}(\ref{item:finite-etale-component-group}), we may thus pass to the fraction field of $\sO$ to assume $\chara k = 0$. \smallskip

Now let $p$ be any good prime for $G$. By \cite[Thm.\ 5.1]{Integral-Springer}, it suffices to consider the nilpotent case. By Lemma~\ref{lemma:mcninch-existence-of-sections}, there is a DVR $A$ of mixed characteristic $(0, p)$ and a pure fiberwise nilpotent section $\widetilde{X} \in (\Lie \bG)(A)$ such that $\widetilde{X}_\eta$ is geometrically conjugate to $X$. Thus by Theorem~\ref{theorem:flat-pure-centralizer-semidirect-product}(\ref{item:finite-etale-component-group}), $Z_G(\widetilde{X})/Z_G(\widetilde{X})^0Z(G)$ has finite etale component group of order invertible on $A$, so indeed $p$ does not divide the order of $Z_G(X)/Z_G(X)^0Z(G)$, as desired.
\end{proof}

\subsection{Conjugacy of fiberwise nilpotent sections}

We now move on to considering a few conjugacy results. First, we show that over a connected reduced base scheme, a pure fiberwise nilpotent section is determined up to etale-local conjugacy by a single fiber.

\begin{prop}\label{prop:pure-sections-determined-by-one-fiber}
Let $S$ be a reduced scheme, and let $G$ be a reductive $S$-group scheme such that $|\pi_1(\sD(G))|$ is invertible on $S$ and for each $s \in S$ the characteristic of $k(s)$ is good for $G_s$.\smallskip

Let $X, Y \in \fg(S)$ be pure fiberwise nilpotent sections, and let $W$ be the set of $s \in S$ such that $X_s$ and $Y_s$ are $G(k(\overline{s}))$-conjugate. Then $W$ is open and closed in $S$ and $\Transp_G(X, Y) \times_S W$ is a $Z_G(X) \times_S W$-torsor.
\end{prop}

\begin{proof}
We may pass as usual to the case that $Z(G)$ is smooth. To show that $W$ is open and closed, we may pass first to the case that $S$ is affine. By spreading out, we may assume further that $S$ is noetherian. Now because $S$ is noetherian it suffices to show that $W$ is closed under both specialization and generization, and to check this we may reduce to the case that $S$ is the spectrum of a DVR. To show that $\Transp_G(X, Y) \times_S W$ is a $Z_G(X) \times_S W$-torsor, it suffices to show that it is smooth, as then it admits sections etale-locally. To show this smoothness, we may also use the valuative criterion of flatness \cite[IV\textsubscript{3}, Thm.\ 11.8.1]{EGA} to reduce that $S$ is the spectrum of a DVR. Thus from now on we assume $S = \Spec A$ for a DVR $A$. By descent, we may even assume that $S$ is complete with algebraically closed residue field. \smallskip

First suppose that $t = s$ is the closed point of $S$. In this case, we wish to show that $X$ and $Y$ are $\Ad_{G(A)}$-conjugate given that $X_s$ and $Y_s$ are $\Ad_{G(k(s))}$-conjugate. Since $G$ is smooth we may assume that $X_s = Y_s$. Now let $\tau_1, \tau_2: \bG_m \to G$ be cocharacters which are fiberwise associated to $X$ and $Y$, respectively; these exist by Lemma~\ref{lemma:flat-pure-centralizer-consequences}(\ref{item:associated-characters-exist-integrally}). Since $Z_G(X)$ is smooth and any two associated cocharacters of $X$ are fiberwise geometrically conjugate by $Z_G(X)$ (by Lemma~\ref{lemma:existence-of-associated-cocharacters}(\ref{item:associated-cocharacters-conjugate})), we may conjugate to assume $(\tau_1)_s = (\tau_2)_s$. But now \cite[Cor.\ B.3.5]{Conrad} implies that $\tau_1$ and $\tau_2$ are $G(A)$-conjugate. Thus after passing to a further $\Ad_{G(A)}$-conjugate of $X$ we may assume that $\tau_1 = \tau_2$. We will denote their common value by $\tau$. \smallskip

Let now $P = P_G(\tau)$, so that $P$ is fiberwise the instability parabolic of both $X$ and $Y$. Consider the orbit map $\Phi: P \to \bigoplus_{n \geq 2} \fg(\tau, n)$ given by $\Phi(g) = (\Ad g)X$, and note that $\Phi$ is flat by Lemma~\ref{lemma:flat-pure-centralizer-consequences}(\ref{item:orbit-map-is-flat}). If $\Omega$ is the (open) image of $\Phi$, then the induced morphism $\Phi: P \to \Omega$ is a $Z_G(X)$-torsor. Note that $Y$ is a section of $\Omega$: it is certainly a section of $\bigoplus_{n \geq 2} \fg(\tau, n)$, so to check that it lies in the open subscheme $\Omega$ it suffices to check this on special fibers. But it is clear that $Y_s \in \Omega_s$ since $Y_s = X_s$. Thus we see that $\Transp_G(X, Y) = \Phi^{-1}(Y)$ is a $Z_G(X)$-torsor. Since $Z_G(X)$ is smooth, also $\Transp_G(X, Y)$ is smooth, so we are done in this case since sections in the special fiber lift to sections over $A$. \smallskip

Next assume that $t = \eta$, so (after possibly extending $A$) $X_\eta$ and $Y_\eta$ are $\Ad_{G(k(\eta))}$-conjugate. We want to show that then $X$ and $Y$ are $\Ad_{G(A)}$-conjugate, and for this it again suffices to show that $\Transp_G(X, Y)$ is a $Z_G(X)$-torsor: by strict henselianity of $A$, it will therefore have a rational point in its special fiber, and this point will lift to a section over $A$. By (\ref{item:instability-parabolic-exists-integrally}), there are parabolic $S$-subgroups $P$ and $Q$ of $G$ such that $P_u$ is the instability parabolic of $X_u$ and $Q_u$ is the instability parabolic of $Y_u$ for all $u \in S$. (Recall our convention that $s$ denotes the closed point of $S$.) Note that $P_\eta$ and $Q_\eta$ are $G(k(\eta))$-conjugate, and we claim that in fact $P$ and $Q$ are $G(A)$-conjugate. \smallskip

To prove this claim, we note that $P$ and $Q$ lie in the same component of the scheme of parabolics (as may be checked generically), so \cite[Exp.\ XXVI, Thm.\ 3.3]{SGA3III} and \cite[Cor.\ 5.2.8]{Conrad} show that this component is isomorphic to the etale sheaf quotient $G/P$, and hence $P$ and $Q$ are etale-locally conjugate. Since $S$ is strictly henselian, we have $(G/P)(S) = G(S)/P(S)$, so it follows that $P$ and $Q$ are $G(A)$-conjugate, as desired. (In fact, the equality $(G/P)(S) = G(S)/P(S)$ is true whenever $S$ is semi-local; see \cite[Exp.\ XXVI, Cor.\ 5.2]{SGA3III}.) \smallskip

Since $P$ and $Q$ are $G(A)$-conjugate, we may replace $Y$ by an $\Ad_{G(A)}$-conjugate to assume $P = Q$. Let $\tau: \bG_m \to P$ be a cocharacter such that $\tau_u$ is almost associated to both $X_u$ and $Y_u$ for all $u \in S$; the fact that such $\tau$ exists for $X$ is established in (\ref{item:instability-parabolic-exists-integrally}), and if $\tau_\eta$ is almost associated to $X_\eta$ then it is also almost associated to $Y_\eta$ by conjugacy. By Lemma~\ref{lemma:mcninch-relative-instability-parabolic}, it follows that $\tau_u$ is almost associated to $Y_u$ for all $u \in S$. Consider the orbit map $\Phi: P \to \bigoplus_{n \geq 2} \fg(\tau, n)$ for $X$ as before. By Theorem~\ref{theorem:flat-pure-centralizer-semidirect-product}(\ref{item:orbit-map-is-flat}), we know that $\Phi$ is flat, and it is a $Z_G(X)$-torsor onto its open image. By Theorem~\ref{theorem:flat-pure-centralizer} (using our assumption that $Z(G)$ is smooth), $Z_G(X)$ is smooth and thus $\Phi$ is also smooth. Note that the section $Y$ lies in the image $\Omega$ of $\Phi$: since $\Omega$ is open this may be checked on the special fiber, where it follows because the orbit of $Y_s$ is open and dense in $(\bigoplus_{n \geq 2} \fg(\tau, n))_s$. Thus the pullback $\Phi^{-1}(Y)$ is a $Z_G(X)$-torsor; this pullback represents $\Transp_G(X, Y)$, so indeed $\Transp_G(X, Y)$ is smooth.
\end{proof}

The following is the main conjugacy result of this paper; it is a somewhat strengthened form of \cite[Thm.\ 1.6.1(c)]{McNinch-local-fields}.

\begin{theorem}\label{theorem:grothendieck-serre}
Let $A$ be a henselian local ring which is either regular or a valuation ring, and let $G$ be a reductive $A$-group scheme such that the residue characteristic is good for $G$. Let $X, Y \in \fg$ be pure fiberwise nilpotent sections. If $Z \subset Z(G)$ is an $A$-flat closed $A$-subgroup scheme such that $Z(G)/Z$ is smooth, then the following are equivalent.
\begin{enumerate}
    \item\label{item:conjugacy-integral} $X$ and $Y$ are $(G/Z)(A)$-conjugate.
    \item\label{item:conjugacy-special} $X_s$ and $Y_s$ are $(G/Z)(k(s))$-conjugate.
    \item\label{item:conjugacy-generic} $X_\eta$ and $Y_\eta$ are $(G/Z)(k(\eta))$-conjugate.
\end{enumerate}
A similar equivalence holds if $u, v \in G(A)$ are pure fiberwise unipotent sections each lying in the unipotent radical of some parabolic $A$-subgroup scheme of $G$. (For instance, this is the case if $|\pi_1(\sD(G))|$ is invertible in $A$.)
\end{theorem}

\begin{proof}
First, it is clear that (\ref{item:conjugacy-integral}) implies (\ref{item:conjugacy-generic}) and (\ref{item:conjugacy-special}). For the other implications, we first reduce to the case that $\sD(G)$ is simply connected. This reduction is very similar in the nilpotent and unipotent cases, and after making this reduction the proofs in the two cases are almost identical, so we will deal mainly with the nilpotent case in what follows.\smallskip

Let $T_0$ be the maximal central torus of $G$. Let $\pi_G: \widetilde{G} \to \sD(G)$ be the map from the universal cover of $\sD(G)$, and embed $Z(\widetilde{G})$ into an $A$-torus $T_1$. Let $\widetilde{G}_1 = \widetilde{G} \times^{Z(\widetilde{G})} (T_0 \times T_1)$, where the map $Z(\widetilde{G}) \to T_0 \times T_1$ is given by the two maps
\[
Z(\widetilde{G}) \to Z(\sD(G)) \to T_0
\]
and the chosen embedding $Z(\widetilde{G}) \to T_1$. Note that there is a natural flat homomorphism $\pi: \widetilde{G}_1 \to G$ with kernel $T_1$, and $\sD(\widetilde{G}_1) = \widetilde{G}$. \smallskip

Let $\widetilde{P}_1$ be the pullback of the instability parabolic $P$ of $X$ to $\widetilde{G}_1$, and let $\widetilde{\fu}_1$ and $\fu$ be the Lie algebras of the unipotent radicals of $\widetilde{P}_1$ and $P$. The map $\widetilde{\fu}_1 \to \fu$ is an isomorphism, so there is a (unique) fiberwise nilpotent lift $\widetilde{X}$ of $X$ in $\widetilde{\fg}_1$, the Lie algebra of $\widetilde{G}_1$. Of course, $Y$ also lifts uniquely to a fiberwise nilpotent section $\widetilde{Y}$ of $\widetilde{\fg}_1$. \smallskip

Note that $\widetilde{X}$ and $\widetilde{Y}$ are automatically pure. To see purity of $\widetilde{X}$, it suffices to mention that the natural map
\[
Z_{\widetilde{G}_1}(\widetilde{X}) \to \pi^{-1}(Z_G(X))
\]
is a closed embedding on fibers with finite cokernel, as one sees by an argument akin to that of the proof of \cite[Lem.\ 3.3(2)]{Integral-Springer}. So purity of $\widetilde{X}$ follows from flatness of $\pi$ and purity of $X$. Thus passing from $G$ to $\widetilde{G}_1$, $Z$ to $\pi^{-1}(Z)$, and $X, Y$ to $\widetilde{X}, \widetilde{Y}$, we may assume that $\sD(G)$ is simply connected. In this case it is clear from Proposition~\ref{prop:pure-sections-determined-by-one-fiber} that $\Transp_G(X, Y)/Z$ is smooth and thus (\ref{item:conjugacy-special}) implies (\ref{item:conjugacy-integral}). \smallskip

The reduction to the case that $\sD(G)$ is simply connected in the unipotent case is very similar, so we will only explain the final parenthetical in the statement of the result. By \cite[Thm.\ 5.1]{Integral-Springer} there is a $G$-equivariant isomorphism $\rho: \sU_G \cong \sN_G$. If $X = \rho(u)$ and $P$ is the instability parabolic of $X$, then $G$-equivariance of $\rho$ implies that $u$ lies in the unipotent radical of $P$; see \cite[Rmk.\ 10]{Mcninch-optimal}. In the remainder of this proof, we will deal only with the nilpotent case; in the unipotent case, the role of the instability parabolic of $X$ is played by the above parabolic $P$. \smallskip

Now assume that (\ref{item:conjugacy-generic}) holds. Note that smoothness of $Z(G)/Z$ implies smoothness of $C_X/Z$ by Lemma~\ref{lemma:flat-pure-centralizer-consequences}. We note that $\Transp_G(X, Y)/Z$ is a $Z_G(X)/Z$-torsor over $A$ which is generically trivial, so it suffices to show that the map $\rm{H}^1(A, Z_G(X)/Z) \to \rm{H}^1(k(\eta), Z_G(X)/Z)$ has trivial kernel. We call attention to the short exact sequences
\begin{gather*}
1 \to R_X \to Z_G(X)/Z \to C_X/Z \to 1 \\
1 \to (C_X/Z)^0 \to C_X/Z \to (C_X/Z)/(C_X/Z)^0 \to 1
\end{gather*}
where $C_X$ and $R_X$ are as in Theorem~\ref{theorem:flat-pure-centralizer-semidirect-product}(\ref{item:nilpotent-centralizer-unipotent-part}). For ease of notation, let $E_X = C_X/Z$. We show first that $\rm{H}^1(A, R_X) = 1$ and then a diagram chase will allow us to pass from $Z_G(X)/Z$ to $E_X$. We then further reduce to showing that $\rm{H}^1(A, E_X^0) \to \rm{H}^1(k(\eta), (E_X^0)_\eta)$ has trivial kernel. \smallskip

First we show $\rm{H}^1(A, R_X) = 1$. Since $A$ is henselian and $R_X$ is smooth, it suffices for the first claim to show that $\rm{H}^1(k(s), (R_X)_s) = 1$, and thus we may assume in any case that $A = k$ is a field. We claim that $R_X$ is $k$\textit{-split}, i.e., there is a filtration $1 = U_0 \subset U_1 \subset \cdots \subset U_n = R_X$ such that each successive quotient $U_{i+1}/U_i$ is $k$-isomorphic to $\bG_a$. If $k$ is perfect, then this is automatic; see \cite[Cor.\ 15.5]{Borel}. Thus we will assume from now on that $k$ is of positive characteristic $p$. \smallskip

Note that $R_X$ is a smooth connected unipotent $k$-group scheme, and it admits a $\bG_m$-action with only nontrivial weights on its Lie algebra (coming from an associated cocharacter for $X$). We claim that this automatically implies that $R_X$ is $k$-split. To see this, we may induct on the length of the central series for $R_X$ to assume that $R_X$ is commutative. Work of Tits shows that any $p$-torsion smooth connected commutative $k$-group scheme admitting a $\bG_m$-action with only nontrivial weights on its Lie algebra is automatically a vector group; see \cite[Thm.\ B.4.3]{CGP}. So if $H \subset R_X$ is the maximal $p$-torsion smooth connected commutative $k$-group scheme, then $H$ is a nontrivial characteristic $k$-subgroup of $R_X$ which is a vector group. Thus induction on $\dim R_X$ proves the claim. \smallskip

Since $R_X$ is $k$-split, the triviality of $\rm{H}^1(k, R_X)$ reduces to that of $\rm{H}^1(k, \bG_a)$. Thus, as explained, we may reduce to showing that the natural map $\rm{H}^1(A, E_X) \to \rm{H}^1(k(\eta), (E_X)_\eta)$ has trivial kernel. \smallskip

Next, we check that if $E$ is a finite etale $A$-scheme then the map $E(A) \to E(k(\eta))$ is surjective. If $A$ is a valuation ring, then this follows from the valuative criterion of properness (since $E$ is finite). If $A$ is a regular local ring, this follows from surjectivity of the map $\pi_1(\Spec k(\eta), \overline{\eta}) \to \pi_1(\Spec A, \overline{\eta})$; see \cite[Exp.\ V, Prop.\ 8.2]{SGA1}. Applying this observation to an $E_X/E_X^0$-torsor $E$ shows that the map $\rm{H}^1(A, E_X/E_X^0) \to \rm{H}^1(k(\eta), (E_X/E_X^0)_\eta)$ has trivial kernel. Looking at the diagram
\[
\begin{tikzcd}
(E_X/E_X^0)(A) \arrow[r] \arrow[d]
    &\rm{H}^1(A, E_X^0) \arrow[r] \arrow[d]
    &\rm{H}^1(A, E_X) \arrow[r] \arrow[d]
    &\rm{H}^1(A, E_X/E_X^0) \arrow[d] \\
(E_X/E_X^0)(k(\eta)) \arrow[r]
    &\rm{H}^1(k(\eta), (E_X^0)_\eta) \arrow[r]
    &\rm{H}^1(k(\eta), (E_X)_\eta) \arrow[r]
    &\rm{H}^1(k(\eta), (E_X/E_X^0)_\eta)
\end{tikzcd}
\]
we may therefore reduce finally to showing that $\rm{H}^1(A, E_X^0) \to \rm{H}^1(k(\eta), E_X^0)$ has trivial kernel. \smallskip

By Theorem~\ref{theorem:flat-pure-centralizer-semidirect-product}(\ref{item:nilpotent-centralizer-unipotent-part}), $C_X/Z(G)$ is smooth and affine with reductive fibers. Since the map $E_X \to C_X/Z(G)$ is a $Z(G)/Z$-torsor and $Z(G)/Z$ is a smooth $A$-group scheme of multiplicative type, also $E_X$ is smooth and affine with reductive fibers. Thus by \cite[Prop.\ 3.1.3]{Conrad}, $E_X^0$ is a reductive group scheme, so our injectivity assertion is precisely the statement of the \textit{Grothendieck--Serre conjecture} for $E_X^0$. When $A$ is a valuation ring, the Grothendieck--Serre conjecture is the main result of \cite{Guo-valuation}. The henselian regular local case is readily deduced from this by induction on the dimension; see \cite[6.6.1]{CTS}.
\end{proof}

\begin{remark}
It seems plausible that Theorem~\ref{theorem:grothendieck-serre} holds whenever the Grothendieck--Serre conjecture holds for $(A, H)$ as $H$ ranges all reductive $A$-subgroup schemes of $G$. (For recent progress on this conjecture, see \cite{Fedorov-Panin}, \cite{Panin}, \cite{Cesnavicius}.) However, to show that $\rm{H}^1(A, R_X) = 0$ we used that $A$ is henselian, and we do not know whether this assumption can be removed in general.
\end{remark}

\bibliography{bibliography}

\end{document}